\documentclass[12pt,reqno]{amsart}

\usepackage{amssymb,amsmath, amsfonts, latexsym}
\usepackage{enumerate}

\newcommand{\rr}{\mathbb{R}}

\newtheorem{lem}{Lemma}[section]

\newtheorem{thm}{Theorem}[section]

\newtheorem{rmk}{Remark}[section]

\numberwithin{equation}{section}

 \numberwithin{equation}{section}

\title[A variational representation and LDP for $G$-Brownian functionals]{A variational representation and large deviations for functionals of $G$-Brownian
motion}

\author{Fuqing Gao}

\address{Fuqing Gao,
School of Mathematics and Statistics, Wuhan University, 430072
Wuhan, China} \email{fqgao@whu.edu.cn}

\thanks{Research supported by the National Natural Science Foundation
of China  (11171262) }

\date{}

\begin{document}

\begin{abstract}
A variational representation for functionals of $G$-Brownian motion
is established by a finite-dimensional approximate
technique. As an application of the variational representation,
we obtain a large deviation principle  for stochastic flows driven
by G-Brownian motion.

\end{abstract}



\maketitle


\section{Introduction}
Peng (\cite{peng-G-brown-1}) proposed $G$-Brownian motion and $G$-expectation.   The stochastic analysis under the $G$-expectation ($G$-stochastic calculus) has had many important progresses in recent years (cf. \cite{peng-book-10} and references therein).   For a connection between Denis and Martini (\cite{Denis-Martini}) and the $G$-stochastic
integration theory of Peng (\cite{peng-G-brown-1}), we refer to  Denis, Hu and Peng (\cite{Denis-Hu-Peng}), and Soner, Touzi and Zhang (\cite{SonerTouziZhang}).
The $G$-stochastic calculus also provides a framework for financial problems with uncertainty about
the volatility and a stochastic method for fully nonlinear PDEs (cf. \cite{Denis-Martini},\cite{peng-G-brown-1}, \cite{SonerTouziZhang-2bsde}).

The purpose of this paper is to establish a variational representation for functionals  of $G$-Brownian
motion and a large deviation principle for stochastic flows driven by $G$-Brownian
motion. We obtain the following variational representation for functionals  of $G$-Brownian
motion:
\begin{equation}\label{Variational-formula-0}
\mathbb E^G ( e^{\Phi(B)} ) = \exp \left\{
 \sup_{\eta\in (M^2(0,T))^d} \mathbb E^G \left( \Phi\left(B^{\eta}\right)
 -H_T^G(\eta)
\right) \right\},
\end{equation}
where $\Phi\in L^1_G(\Omega_T)$ bounded,  $\{B_t,t\in[0,T]\}$ is a
$G$-Brownian motion and $\{\langle B\rangle_t,t\in[0,T]\}$ is its
quadratic variation process,  $B_t^{\eta}=B_t+\int_0^t\eta_s d\langle B\rangle_s$ and $H_T^G(\eta)= \frac{1}{2} \sum_{i,j=1}^d\int_0^T\eta_s^i\eta_s^j d\langle B\rangle_s^{ij}$. The definitions of $\mathbb E^G$, $L^1_G(\Omega_T)$,
$M^2(0,T)$, the $G$-Brownian motion and the quadratic variation
process will be given in  Section 2.  As an application of the variational
representation, we obtain a large deviation principle  for
stochastic flows driven by G-Brownian motion.

In the classical case, a variational representation of functionals
of finite dimensional Brownian motion was first obtained by Bou\'e
and Dupuis (\cite{BoueDupuis-AP-98}).  Chen and Xiong
(\cite{ChenXiong-10}) considered the variational representations
under a $g$-expectation which is defined by a backward stochastic
differential equation. The variational representations have been
shown to be useful in deriving various asymptotic results in large
deviations (cf.   \cite{BoueDupuis-AP-98},
\cite{BDM-infty-AP-08}, \cite{BDM-Bernoulli-10},
 \cite{DupuisEllis} and  \cite{RenZhang}) and functional inequalities (cf.  \cite{borell-PA-00}).

Under  $G$-expectation, the complicated measurable selection
technique in \cite{BoueDupuis-AP-98} cannot be used  and the Clark-Ocone formula
is not available.  In this paper, we will develop  finite-dimensional approximate
technique under $G$-expectation. We prove that a finite-dimensional functional for $G$-Brownian motion can be approximated by  a sequence of $G$- stochastic differential
equations, which plays an important role in the proof of the upper bound. The lower bound will be proved by the $G$-Girsanov
transformation (cf. \cite{XuShangZhang}) and bounded approximation.  In particular,   this also provides a new proof for the variational representations  of Bou\'e and
Dupuis.

The remainder of the paper is organized as follows. In Section 2 we
recall some basic conceptions and results under  $G$-framework.
The variational representation is proved in Section 3. An abstract large deviation principle for functionals of $G$-Brownian motion is presented in Section 4.  A
large deviation principle of stochastic flows driven by $G$-Brownian
motion is established in Section 5.

\section{$G$-expectation and $G$-Brownian motion}

In this section,   we briefly recall some basic conceptions and
results about $G$-expectation and $G$-Brownian motion (see
\cite{Denis-Hu-Peng},\cite{peng-G-brown-1},  \cite{peng-G-brown-d}
and \cite{peng-book-10} for details).

\subsection{Sublinear
expectation}

Let $\Omega$ be a given set and let $\mathcal{H}$ be a
linear space of real valued functions defined on $\Omega$ with $c\in
\mathcal{H}$ for all constants $c$, and satisfying:
if $X_{i}\in \mathcal{H}$, $i=1,\cdots,d$,
then
$$
\varphi(X_{1},\cdots,X_{d})\in \mathcal{H}\text{,\  \ for all }\varphi \in
lip(\mathbb{R}^{d}),
$$
where $lip(\mathbb{R}^{d})$ is the space of all bounded and Lipschitz
continuous functions on $\mathbb{R}^{d}$.

A sublinear
expectation $\mathbb{\hat{E}}$ on $\mathcal{H}$ is a functional
$\mathbb{\hat{E}}:\mathcal{H}\mapsto \mathbb{R}$ satisfying the following
properties:

{\it Monotonicity:}  If $X\geq Y$ then
$\mathbb{\hat{E}}(X)\geq \mathbb{\hat{E}}(Y)$;

{\it Constant preserving:} $\mathbb{\hat{E}}(c)=c$;

{\it Sub-additivity:}  $\mathbb{\hat{E}}(X)-\mathbb{\hat{E}}(Y)\leq \mathbb{\hat{E}}(X-Y)$;

{\it Positive homogeneity:}  $\mathbb{\hat{E}}(\lambda X)=\lambda \mathbb{\hat{E}
}(X)$,$\  \  \forall \lambda \geq0$.

\noindent
The triple $(\Omega
,\mathcal{H},\mathbb{\hat{E}})$ is called a sublinear expectation space.
$X\in \mathcal{H}$ is called a random variable in $(\Omega,\mathcal{H})$.

A $m$-dimensional random vector $X=(X_{1},\cdots,X_{m})$ is said to
be independent of another $n$-dimensional random vector
$Y=(Y_{1},\cdots,Y_{n})$ if
$$
\mathbb{\hat{E}}(\varphi(X,Y))=\mathbb{\hat{E}}(\mathbb{\hat{E}}%
(\varphi(X,y))_{y=Y}),\  \  \text{for }\varphi \in lip(\mathbb{R}^{m}%
\times \mathbb{R}^{n}).
$$

Let $X_{1}$ and $X_{2}$ be two $d$--dimensional random vectors defined
respectively in {sublinear expectation spaces }$(\Omega_{1},\mathcal{H}%
_{1},\mathbb{\hat{E}}_{1})${ and }$(\Omega_{2},\mathcal{H}_{2},\mathbb{\hat
{E}}_{2})$. They are called identically distributed, denoted by $X_{1}\sim
X_{2}$, if
$$
\mathbb{\hat{E}}_{1}(\varphi(X_{1}))=\mathbb{\hat{E}}_{2}(\varphi
(X_{2})),\  \  \  \forall \varphi \in C_{b.Lip}(\mathbb{R}^{n}).
$$

A $d$-dimensional
random vector $X=(X_{1},\cdots,X_{d})$ in a sublinear expectation space
$(\Omega,\mathcal{H},\mathbb{\hat{E}})$ is called $G$-normal distributed if
for each $a\,$, $b\geq0$ we have
$
aX+b\bar{X}\sim \sqrt{a^{2}+b^{2}}X,
$
where $\bar{X}$ is an independent copy of $X$. The letter $G$ denotes the
function $G:\mathbb{S}_{d}\mapsto \mathbb{R}$:
$
G(A):=\frac{1}{2}\mathbb{\hat{E}}(( AX,X)),
$
where $\mathbb{S}_d$ is the collection of $d\times d$ symmetric
matrices and $(\cdot,\cdot)$ denotes the inner product in $\mathbb
R^d$, i.e., $ ( x,y):=\sum_{i=1}^dx^iy^j$ for any
$x=(x^1,\cdots,x^d)$, $y=(y^1,\cdots,y^d)\in\mathbb R^d$.

Let $d$-dimensional random vector $X=(X_{1},\cdots,X_{d})$ in a
sublinear expectation space $(\Omega,\mathcal{H},\mathbb{\hat{E}})$
be $G$-normal distributed.  For each $\varphi\in lip (\mathbb R^d)$,
set
$
u(t,x)=\mathbb{\hat{E}}\left(\varphi(x+\sqrt{t}X)\right)$, $t\geq 0$,
$x\in\mathbb R^d$.
Then $u(t,x)$ is the unique  viscosity solution of the following
equation,
\begin{equation}
\begin{aligned}
\frac{\partial u}{\partial t} - G\left(D_x^2 u\right) =& 0,~~t\geq
0,~ x\in\mathbb R^d,\quad
u(0,x) = \varphi(x),
\end{aligned}
\end{equation}
where $D_x^2 u=(\partial_{x^{i}x^{j}}^{2}u)_{i,j=1}^{d}$ is the
Hessian matrix of $u$.

The inner product in $\mathbb S_d$ is defined by
$
(A_1,A_2)=\sum_{i,j=1}^d a_1(i,j)a_2(i,j)$ for $A_1=(a_1(i,j))_{d\times
d }$ ,$A_2=(a_2(i,j))_{d\times d }$.
Then the map $G:  \mathbb{S}_d\mapsto \rr$ is a monotonic and
sublinear function, i.e., for $A,\bar{A} \in \mathbb{S}_d$,
\begin{equation}\label{mono}
\left \{
\begin{array}{l}
G(A+\bar{A})   \leq G(A)+G(\bar{A}), \\
G(\lambda A)   =\lambda G(A),\    \mbox{  for all  } \lambda \geq0,\\
G(A)  \geq G(\bar{A}),\   \mbox{  if  } A\geq \bar{A}.
\end{array}
\right.
\end{equation}

For a monotonic and sublinear function $G:  \mathbb{S}_d\mapsto \rr$
given, there exists a bounded, convex and closed subset
$\Sigma\subset \mathbb S_d^+$ (the non-negative elements of $\mathbb
S_d$) such that
$
G(A) = \frac{1}{ 2}\sup_{\sigma\in \Sigma}(A,\sigma).
$
Throughout this paper, we assume that there exist constants
$0<\underline{\sigma}\leq \overline{\sigma}<\infty$ such that
\begin{equation} \label{basic-condition-2}
\Sigma\subset \left\{\sigma\in\mathbb S_d;~\underline{\sigma}
I_{d\times d}\leq\sigma\leq \overline{\sigma} I_{d\times d}\right\}.
\end{equation}

\subsection{$G$-Brownian motion and $G$-expectation}

Let $\Omega$ denote the space of all
$\mathbb{R}^{d}$-valued continuous paths $\omega: ~(0,+\infty)\ni
 t\longmapsto\omega_{t}\in\mathbb R^d$, with $\omega_{0}=0$.
Let $\mathcal{B}(\Omega)$, $\mathcal{M}$, $L^{0}(\Omega)$, $B_{b}(\Omega)$ and $C_{b}(\Omega)$ denote respectively the Borel $\sigma$-algebra of $\Omega$,
the collection of all probability measure on $\Omega$,
the space of all $\mathcal{B}(\Omega)$-measurable real functions,
all bounded elements in $L^{0}(\Omega)$
and all continuous elements in $B_{b}(\Omega)$.
For each $t\in \lbrack0,\infty)$, we also denote
$\Omega_{t}:=\{ \omega_{\cdot \wedge t}:\omega \in \Omega \}$; $\mathcal{F}_{t}:=\mathcal{B}(\Omega_{t})$;
$L^{0}(\Omega_{t})$: the space of all $\mathcal{B}(\Omega_{t}
)$-measurable real functions; $B_{b}(\Omega_{t}):=B_{b}(\Omega)\cap L^{0}(\Omega_{t})$; $C_{b}(\Omega_{t}):=C_{b}(\Omega)\cap L^{0}(\Omega_{t})$.

For each $t>0$,  set
$$
L_{ip}(\Omega_t):=\left\{\varphi\left(\omega_{t_1},\omega_{t_2},\cdots
\omega_{t_n}\right):
 n\geq1, t_1,\cdots ,t_n\in[0,t], \varphi\in
lip((\mathbb{R}^d)^n))\right\}.
$$
Define $
L_{ip}(\Omega):=\bigcup_{n=1}^{\infty}L_{ip}(\Omega_n)\subset C_{b}(\Omega)$.

Let
$G:\mathbb{S}_{d}\mapsto \mathbb{R}$ be a given monotonic and sublinear
function. A continuous process $\{B_{t}(\omega)\}_{t\geq0}$ in a sublinear expectation
space $(\Omega,\mathcal{H},\mathbb{E}^G)$ is called a $G$-Brownian motion
if it has stationary and independent increments,      $B_0=0$,   $B_{1}$ is $G$-normal distributed and $\mathbb{E}^G(B_{t})=-\mathbb{E}^G(-B_{t})=0$ for $t\geq0$.

The
topological completion of $L_{ip}(\Omega_t)$ (resp.
$L_{ip}(\Omega)$) under the Banach norm $\|\cdot\|_{p,G}:=(\mathbb{E}^G(|\cdot|^p))^{1/p}$ is
denoted by $L_{G}^{p}(\Omega_t)$ (resp. $L_{G}^{p}(\Omega)$), where $p\geq 1$. $\mathbb{E}^G(\cdot)$
can be extended uniquely to a sublinear
expectation on $L_{G}^{1}(\Omega)$. We denote also by $\mathbb{E}^G$ the
extension. It is proved in \cite{Denis-Hu-Peng} that $L^{0}(\Omega)\supset
L_{G}^{p}(\Omega)\supset C_{b}(\Omega)$, and there exists a weakly
compact family $\mathcal{P}$ of probability measures defined on $(\Omega
,\mathcal{B}(\Omega))$ such that
$
\mathbb{E}^G(X)=\sup_{P\in \mathcal{P}}E_{P}(X)$ for $ X\in
C_{b}(\Omega).$
$\mathbb E^G(\cdot)$ has the following regularity (\cite{Denis-Hu-Peng}):
For each $\{X_{n} \}_{n=1}^{\infty}$ in $C_{b}(\Omega)$ with
$X_{n}\downarrow0$ \ on $\Omega$,   $\mathbb E^G (X_{n}%
)\downarrow0$. We also denote
$$
\overline{\mathbb{E}}^G(X)=\sup_{P\in \mathcal{P}}E_{P}(X),~~~X\in
L^0(\Omega).
$$

The natural Choquet capacity  associated with $\mathbb E^G$ is defined by
$$
c^G(A):=\sup_{P\in \mathcal{P}}P(A),\  \ A\in \mathcal{B}(\Omega).
$$
A set $A\subset \Omega$ is polar if $c^G(A)=0$. A property holds
\textquotedblleft quasi-surely\textquotedblright \ (q.s.) if it
holds outside a polar set. A mapping $X$ on $\Omega$ with values in
a topological space is said to be quasi-continuous (q.c.) if for any
$\varepsilon>0 $, there exists an open set $ O$ with
$c^G(O)<\varepsilon$  such that $X|_{O^{c}}$ is continuous.  $ L
_{G}^{p}(\Omega)$ also has the following characterization
(\cite{Denis-Hu-Peng}):
$$
\begin{aligned}
L_{G}^{p}(\Omega)=\big\{X\in
{L}^{0}(\Omega);~ &\lim_{n\to\infty}\overline{\mathbb{E}}^G(|X|^{p}I_{\{|X|\geq n\}})=0, \\
&\text{ and
}X\text{ is }c^G\text{-quasi surely continuous}\big\}.
\end{aligned}
$$

Let us recall the representation theorem of $G$-expectation. For a
monotonic and sublinear function $G:  \mathbb{S}_d\mapsto \rr$
given, there exists a bounded, convex and closed subset
$\Sigma\subset \mathbb S_d^+$  such that
$
G(A) = \frac{1}{ 2}\sup_{\sigma\in \Sigma}(A,\sigma).
$
Set  $\Gamma:= \{\gamma=\sigma^{1/2}, \sigma\in\Sigma\}$. Let $P$ be
the Wiener measure on $\Omega$. Let
$\mathcal{A}_{0,\infty}^{\Gamma}$ be the collection of all
$\Gamma$-valued $\{\mathcal{F}_t,t\geq 0\} $-adapted processes on
the interval $[0,+\infty)$, i.e., $\{\theta_t,t\geq 0\}\in
\mathcal{A}_{0,\infty}^{\Gamma}$ if and only if $\theta_t$ is
$\mathcal F_t$ measurable and $\theta_t\in \Gamma$ for each $t\geq
0$, and let $P_{\theta}$ be the law of the process
$\{\int_{0}^{t}\theta_sd\omega_s, t\geq0\}$ under the Wiener measure
$P$. The representation theorem of
$G$-expectation(\cite{Denis-Hu-Peng}) is stated as follows: for
all~$X\in L_G^1(\Omega)$
\begin{equation}\label{Girsanov-formula-proof-eq-1}
\mathbb E^G(X)=\sup_{\theta\in\mathcal{ A
}_{0,\infty}^{\Gamma}}E_{P_\theta}(X).
\end{equation}

\subsection{$G$-Stochastic integral and quadratic variation process}

Given $T>0$. For $p \in [1,\infty)$,  let $M_G^{p,0}(0,T)$ denote the space of  $\mathbb R$-valued
piecewise constant processes
$$
\eta_t= \sum_{i= 0}^{n-1} \eta_{t_i}1_{[t_i, t_{i+1})}(t)
$$
where $\eta_{t_i}\in   L^p_G(\Omega_{t_i})$, $0=t_0<t_1<\cdots<t_{n}=T$. For $\eta\in M_G^{p,0}(0,T)$, $j=1,\cdots,d$,
the $G$-stochastic integral is defined by
$$
I^j(\eta):=\int_0^T \eta_s d B_s^j := \sum_{i= 0}^{n-1}
\eta_{t_i}(B^j_{t_{i+1}}-B^j_{t_i}).
$$
Let $M_G^p(0,T)$ be the closure of $M_G^{p,0}(0,T)$ under the norm:
$
\|H\|_{M_G^p(0,T)}^p :=  \mathbb E^G\left(\int_0^T|\eta_t|^pdt\right) .
$
Then the mapping $I^j: M^{2,0}_G(0,T)\rightarrow L^2_G(\Omega_T)$ is
continuous, and so it can be continuously extended to $M^2_G(0,T)$.

For any $\eta=(\eta^1,\cdots,\eta^d)\in(M^{2}_G(0,T))^d$, define
$$
\int_0^T \eta_s dB_s=\sum_{i=1}^d\int_0^T \eta_s^i  d B_s^i .
$$

The quadratic variation process of
$G$-Brownian motion is defined by
$$
\langle B\rangle_t := (\langle B\rangle_t ^{ij})_{1\leq i,j\leq d}=\left(  B_t^iB_t^j -2\int_0^t
B_s^i dB_s^j\right)_{1\leq i,j\leq d},\ \ \
\ \ 0  \le t \le T.
$$
$\langle B\rangle_t$ is a
$\mathbb S_d$-valued process with stationary and independent increments.

For any $1\leq i,j\leq d$, define a mapping $M_{G} ^{1,0}(0,T)\mapsto L_G^{T}(\Omega_{1})$
as follows:
$$
Q^{ij}_{0,T}(\eta)=\int_{0}^{T}\eta_sd\left \langle B \right
\rangle^{ij} _{s}:=\sum_{k=0}^{n-1}\eta_{t_k}(\left \langle
B \right \rangle^{ij} _{t_{k+1}}-\left \langle
B \right \rangle ^{ij}_{t_{k}}).
$$
Then $Q^{ij}_{0,T}$ can be uniquely extended to $M_{G}^{1}(0,T)$. We
still denote this mapping by
$$
\int_{0}^{T}\eta_sd\left \langle B\right \rangle^{ij}
_{s}=Q^{ij}_{0,T} (\eta),\  \  \eta\in M_{G}^{1}(0,T).
$$

For $\eta=(\eta^1,\cdots,\eta^d)\in(M^{1}_G(0,T))^{d}$, define
$$
\int_0^T \eta_s d\left \langle B\right
\rangle_s=\left(\sum_{j=1}^d\int_0^T \eta_s^{j }d\left \langle B\right
\rangle^{ij}_s\right)_{d\times 1}.
$$
and for $\eta=(\eta^{ij})_{d\times d}\in(M^{1}_G(0,T))^{d\times d}$, define
$$
\int_0^T \eta_s d\left \langle B\right \rangle_s=\sum_{i,j=1}^d\int_0^T
\eta_s^{ij}d\left \langle B\right \rangle^{ij}_s.
$$

\subsection{$G$-Girsanov formula}

In one dimensional case, under a strong Novikov-type condition, Xu, Shang and Zhang
(\cite{XuShangZhang}) obtained a Girsanov formula under $G$-expectation
based on the martingale characterization theorem (\cite{XuZhang}) of
$G$-Brownian motion. A multi-dimensional version of the $G$-Girsanov formula is presented in \cite{Osuka}.

For $\eta=(\eta^1,\cdots,\eta^d)\in (M_G^2(0,T))^d$ satisfying the following strong Novikov condition:
\begin{equation}\label{novikov-c}
\mathbb E^G\left(\exp\left\{\frac{1}{2}(1+\epsilon))\sum_{i,j=1}^d\int_0^T \eta_s^i\eta_s^jd\langle B\rangle_s^{ij}\right\}\right)<\infty, \quad \mbox{ for some } \epsilon>0,
\end{equation}
we define
\begin{equation}\label{Girsanov-formula-proof-eq-2}
\mathcal E_t^\eta=\exp\left\{\int_0^t \eta_s dB_s
-\frac{1}{2}\sum_{i,j=1}^d\int_0^t\eta_s^i\eta_s^jd\langle B\rangle_s^{ij}\right\}, ~~0\leq t\leq T.
\end{equation}
Then $\mathcal E_t^\eta$ is quasi-continuous. By the  condition (\ref{novikov-c}), $\mathcal E_t^\eta$, $t\in[0,T]$
is a martingale under each $P_\theta$, and  $\mathcal E_t^\eta\in
L_G^1(\Omega_t)$, $t\in[0,T]$ (cf. \cite{XuShangZhang}). Set
\begin{equation}\label{Girsanov-formula-proof-eq-3}
dP_{\theta,\eta}= \mathcal E_T^\eta dP_\theta,
\end{equation}
and
\begin{equation}\label{Girsanov-formula-proof-eq-4}
 \mathbb E^{G,\eta}(X)=\sup_{\theta\in\mathcal{ A
}_{0,\infty}^{\Gamma}}P_{\theta,\eta}(X),~~X\in L_{ip}(\Omega_T).
\end{equation}
Let $L_{G,\eta}^1(\Omega_t)$ be the completion of $(L_{ip}(\Omega_t),
\mathbb E^{G,\eta}(|\cdot|))$. Then $ \mathbb E^{G,\eta}$ can be extended
to $L_{G,\eta}^1(\Omega_T)$, and the following  G-Girsanov formula holds:

\subsubsection*{G-Girsanov formula}(\cite{XuShangZhang}, \cite{Osuka})  Under the condition (\ref{novikov-c}),
$$
B_t^{-\eta}:=B_t-\int_0^t \eta_s d\langle B  \rangle_s,~~t\in [0,T]
$$
is a $G$- Brownian motion under $\mathbb E^{G,\eta}$.

\section{A variational representation for functionals of $G$-Brownian motion}

The main result in this section is the following variation representation for
functionals of $G$-Brownian motion.

\begin{thm}\label{G-variation-representation-thm} Let $\Phi\in L^1_G(\Omega_T)$ be bounded. Then
\begin{equation}\label{G-variation-representation-thm-eq-1}
\begin{aligned}
\mathbb E^G ( e^{\Phi(B)} ) =& \exp \left\{
 \sup_{\eta\in (M_G^2(0,T) )^d} \mathbb E^G \left( \Phi\left(B^\eta\right)
 -H_T^G(\eta)\right)  \right\}\\
 =& \exp \left\{
 \sup_{\eta\in (M_{G}^{2,0}(0,T) )^d\cap \mathcal B_b(\Omega_T)} \mathbb E^G \left( \Phi\left(B^\eta\right)
 -H_T^G(\eta)\right) \right\},
\end{aligned}
\end{equation}
where $
B_t^{\eta}:=B_t+\int_0^t \eta_s d\langle B  \rangle_s$,
\begin{equation}\label{G-variation-representation-thm-eq-2}
H_t^G(\eta)= \frac{1}{2} \sum_{i,j=1}^d\int_0^t\eta_s^i\eta_s^j d\langle B\rangle_s^{ij},~~t\in [0,T].
\end{equation}

\end{thm}

\begin{rmk}(1). The following proof of Theorem \ref{G-variation-representation-thm} is not depend on the representation for ordinary Brownian motion, the  variational formula of the relative entropy, the measurable selection technique
and the Clark-Ocone formula.   In particular, this also gives a  new proof of the representation for ordinary Brownian motion in \cite{BoueDupuis-AP-98}.

(2). Notice that the $G$-expectation is the supremum of
a collection of expectations so that the canonical map in Wiener space is a  martingale under the expectations. If the representation for ordinary Brownian motion can be extended to continuous martingales, then   Theorem \ref{G-variation-representation-thm} can be obtained from this extension.  But the extension is not available.


\end{rmk}

\begin{lem}\label{V-formula-lem-1} (cf. \cite{peng-book-10}, \cite{SonerTouziZhang})
Let $\eta_s^{ij}\in M_{G}^2(0,T)$, $\eta_s^{ij}=\eta_s^{ji}$, $i,j=1,\cdots, d$  be given   and set
$$
M_t=2\int_0^t G(\eta_s)ds-\int_0^t \eta_s d\langle B\rangle_s,~~t\in[0,T].
$$
Then $M_t\geq 0$, q.s.  for all $t\in[0,T]$. In particular, $t\to M_t$ is increasing.
\end{lem}

\begin{proof}  Take a sequence  $\eta^{(N)}\in (M_G^{2,0}(0,T))^{d\times d}$ , where
$$
\eta_s^{(N)}=\sum_{k=1}^N\eta^{(N)}_{t_{k-1}^{(N)}}I_{[t_{k-1}^{(N)},t_k^{(N)})}(s),~~0=t_0^{(N)}<t_1^{(N)}<\cdots<t_{n_N}^{(N)}=T,
$$
such that
$$
\lim_{N\to\infty} \max_{1\leq i,j\leq d}{\mathbb  E}^G\left(\int_0^T |(\eta_s^{(N)})^{ij}-\eta_s^{ij}|^2\right)ds=0.
$$
Then
$$
\begin{aligned}
& \mathbb E^G\left(\int_0^T| G(\eta_s)- G(\eta_s^{(N)})|\right)ds\\
\leq &  \mathbb E^G\left(\int_0^T\max\{ G(\eta_s-\eta_s^N),G(\eta_s^{(N)}-\eta_s)\}\right)ds\\
= &\frac{1}{2}\mathbb E^G\left(\int_0^T \sup_{\sigma\in \Sigma}| (\eta_s-\eta_s^{(N)},\sigma)| \right)ds\to 0 \mbox{ as } N\to\infty,
\end{aligned}
$$
which yields that for any $t\in [0,T]$,
$$
\lim_{N\to\infty}\mathbb E^G\left(\left|M_t-\left(2\int_0^t G(\eta_s^{(N)})ds-\int_0^t \eta_s^{(N)} d\langle B\rangle_s\right)\right|\right)=0
$$
Thus, it is sufficient to consider the case $\eta^{ij}\in M_G^{2,0}(0,T)$, $\eta_s^{ij}=\eta_s^{ji}$, $i,j=1,\cdots, d$,   i.e.,
$$
\eta_s=\sum_{k=1}^N\eta_{t_{k-1}}I_{[t_{k-1},t_k)}(s),~~0=t_0<t_1<\cdots<t_N=T.
$$
In this case, we have
$$
\begin{aligned}
M_t=&\sum_{k=1}^N \left(2G\left(\eta_{t_{k-1}}(t_k-t_{k-1})\right)-\left(\eta_{t_{k-1}},\langle B\rangle_{t_k}- \langle B\rangle_{t_{k-1}}\right)\right)\\
=&\sum_{k=1}^N \left(\sup_{\sigma\in \Sigma}  \left(\eta_{t_{k-1}}(t_k-t_{k-1}),\sigma\right)-\left(\eta_{t_{k-1}},\langle B\rangle_{t_k}- \langle B\rangle_{t_{k-1}}\right)\right)\\
=&\sum_{k=1}^N  (t_k-t_{k-1})\left(\sup_{\sigma\in \Sigma}  \left(\eta_{t_{k-1}}, \sigma\right)-\left(\eta_{t_{k-1}},\frac{\left(\langle B\rangle_{t_k}- \langle B\rangle_{t_{k-1}}\right)}{t_k-t_{k-1}}\right)\right)\geq 0.
\end{aligned}
$$

\end{proof}

\begin{proof}  {\bf The proof of the lower bound of Theorem \ref{G-variation-representation-thm}.}

{\it Step 1. Bounded case.} \quad Let $\eta$ be bounded. Then by the G-Girsanov-formula,
$$
\begin{aligned}
\mathbb E^G ( e^{\Phi(B)} ) =& \mathbb E^{G,-\eta}( e^{\Phi(B^{\eta})} )
=\mathbb E^G \left( e^{\Phi(B^{\eta})} \exp\left\{- \int_0^T
\eta_sdB_s-H_T^G(\eta)\right\}
\right).
\end{aligned}
$$
By Jensen's inequality, we have
$$
\begin{aligned}
\log\mathbb E^G ( e^{\Phi(B)} ) \geq & \mathbb E^G \left(\Phi(B^{\eta})
-\int_0^T \eta_s dB_s-H_T^G(\eta)\right)
=\mathbb E^G \left(\Phi(B^{\eta}) -H_T^G(\eta) \right).
\end{aligned}
$$

{\it Step 2. General case.}\quad Choose a
sequence $\{\Phi_n,n\geq 1\}$ of uniformly bounded and Lipschitz
continuous functions such that
$$
\lim_{n\to\infty}\mathbb E^G(|\Phi_n-\Phi|^2)=0.
$$
For $\eta\in (M_G^2(0,T) )^d$ given,  we can find a sequence
$\{\eta^{(m)},m\geq 1\}\subset (M_G^{2,0}(0,T))^d\cap \mathcal B_b(\Omega_T)$ such that
$$
\lim_{m\to\infty}\mathbb E^G\left(\int_0^T|\eta_t^{(m)}-\eta_t|^2
dt\right)=0.
$$
Then
$$
\lim_{m\to\infty}\mathbb E^G\left(\left| H_T^G(\eta^{(m)})
- H_T^G(\eta)
\right|\right)=0,
$$
and for each $n\geq 1$, by Lipschitz continuity of $\Phi_n$, we also have
$$
\begin{aligned}
&\lim_{m\to\infty} \mathbb E^G\left(\left|\Phi_n\left(B^{\eta}\right)
-\Phi_n\left(B^{\eta^{(m)}}\right)  \right|\right)
\leq  l(n,\bar{\sigma}) \lim_{m\to\infty}\mathbb E^G\left( \int_0^T
|\eta_t^{(m)}-\eta_t| dt  \right)   =0
\end{aligned}
$$
where $l(n,\bar{\sigma})$ is a constant independent of $m$. Therefore, in order to prove
$\mathbb E^G ( e^{\Phi(B)} ) \geq \mathbb E^G \left( \Phi\left(B^{\eta}\right) -H_T^G(\eta) \right)$,
it only remains to verify
\begin{equation}\label{lower-b-eq-1}
\lim_{n\to\infty}\sup_{m\geq 1}\mathbb E^G\left(\left|\Phi_n\left(B^{\eta^{(m)}}\right)
-\Phi\left(B^{\eta^{(m)}}\right)  \right|\right)=0.
\end{equation}

Fix $\epsilon>0$. Let $M\in (0,\infty)$ be an uniform upper bound $|\Phi_n|,|\Phi|$, $n\geq 1$. Then
$$
\begin{aligned}
& \mathbb E^G\left(\left|\Phi_n\left(B^{\eta^{(m)}}\right)
 -\Phi\left(B^{\eta^{(m)}}\right)  \right|\right)\\
=&  \sup_{\theta\in\mathcal{ A
}_{0,\infty}^{\Gamma}}E_{P_\theta}\left(\left|\Phi_n\left(B^{\eta^{(m)}}\right)
-\Phi\left(B^{\eta^{(m)}}\right)  \right|\right)\\
\leq & 2M \sup_{\theta\in\mathcal{ A
}_{0,\infty}^{\Gamma}} P_\theta \left(\left|\Phi_n\left(B^{\eta^{(m)}}\right)
-\Phi\left(B^{\eta^{(m)}}\right)  \right|>\epsilon\right)+\epsilon.
\end{aligned}
$$
Therefore, we only need to show that for any $\epsilon>0$,
\begin{equation}\label{lower-b-eq-2}
\lim_{n\to\infty}\sup_{m\geq 1}\sup_{\theta\in\mathcal{ A
}_{0,\infty}^{\Gamma}} P_\theta \left(\left|\Phi_n\left(B^{\eta^{(m)}}\right)
-\Phi\left(B^{\eta^{(m)}}\right)  \right|>\epsilon\right)=0.
\end{equation}
For any $N\in(1,\infty)$,
$$
\begin{aligned}
& \sup_{m\geq 1}\sup_{\theta\in\mathcal{ A
}_{0,\infty}^{\Gamma}} P_\theta \left(\left|\Phi_n\left(B^{\eta^{(m)}}\right)
-\Phi\left(B^{\eta^{(m)}}\right)  \right|>\epsilon\right)\\
= & \sup_{\theta\in\mathcal{ A
}_{0,\infty}^{\Gamma}} P_\theta \left(I_{\left\{\left|\Phi_n\left(B^{\eta^{(m)}}\right)
-\Phi\left(B^{\eta^{(m)}}\right)  \right|>\epsilon\right\} }
\mathcal E_T^{-\eta^{(m)}}(\mathcal E_T^{-\eta^{(m)}})^{-1}\right)\\
\leq & N \sup_{\theta\in\mathcal{ A
}_{0,\infty}^{\Gamma}}E_{P_\theta}\left(I_{\left\{\left|\Phi_n\left(B^{\eta^{(m)}}\right) -\Phi\left(B^{\eta^{(m)}}\right)  \right|>\epsilon\right\} }\mathcal E_T^{-\eta^{(m)}} \right)\\
&+\sup_{\theta\in\mathcal{ A
}_{0,\infty}^{\Gamma}}E_{P_\theta}\left(I_{\left\{\left|\Phi_n\left(B^{\eta^{(m)}}\right)-\Phi\left(B^{\eta^{(m)}}\right)  \right|>\epsilon\right\} }
I_{\{\mathcal E_T^{-\eta^{(m)}}\leq 1/N\}}\right).
\end{aligned}
$$
By Chebyshev's inequality and the G-Girsanov-formula, for all $m\geq 1$,
$$
\begin{aligned}
&\sup_{\theta\in\mathcal{ A
}_{0,\infty}^{\Gamma}}E_{P_\theta}\left(I_{\left\{\left|\Phi_n\left(B^{\eta^{(m)}}\right) -\Phi\left(B^{\eta^{(m)}}\right)  \right|>\epsilon\right\} }\mathcal E_T^{-\eta^{(m)}} \right)\\
\leq &  \frac{1}{\epsilon^2}\sup_{\theta\in\mathcal{ A
}_{0,\infty}^{\Gamma}}E_{P_\theta} \left(\left|\Phi_n\left(B^{\eta^{(m)}}\right) -\Phi\left(B^{\eta^{(m)}}\right)  \right|^2\mathcal E_T^{-\eta^{(m)}} \right)\\
=&  \frac{1}{\epsilon^2}\mathbb E^G(|\Phi_n-\Phi|^2)\to0 \mbox{ as } n\to\infty.
\end{aligned}
$$
We also have that
$$
\begin{aligned}
&\sup_{m\geq 1}\sup_{\theta\in\mathcal{ A
}_{0,\infty}^{\Gamma}}E_{P_\theta}\left(I_{\left\{\left|\Phi_n\left(B^{\eta^{(m)}}\right) -\Phi\left(B^{\eta^{(m)}}\right)  \right|>\epsilon\right\} }
I_{\{\mathcal E_T^{-\eta^{(m)}}\leq 1/N\}}\right) \\
\leq &\frac{1}{\log N}\sup_{m\geq 1} \sup_{\theta\in\mathcal{ A
}_{0,\infty}^{\Gamma}}E_{P_\theta}\left(  -\log \mathcal E_T^{-\eta^{(m)}}\right)
=  \frac{1}{ \log N} \sup_{m\geq 1}\mathbb E^G\left( H_T^G(\eta^{(m)})\right)\to 0 \mbox{ as } N\to\infty.
\end{aligned}
$$
(\ref{lower-b-eq-2}) is valid. Hence,
\begin{equation}\label{lower-b-eq-3}
\mathbb E^G ( e^{\Phi(B )} )\geq \exp \left\{
 \sup_{\eta\in (M_G^2(0,T) )^d} \mathbb E^G \left( \Phi\left(B^{\eta}\right)
 - H_T^G(\eta)
\right) \right\}.
\end{equation}

 {\bf The proof of the upper bound Theorem \ref{G-variation-representation-thm}.}

First, let us consider
$\Phi=f(B_{t_1},\cdots,B_{t_n})$, where $f\in lip((\mathbb
R^d)^n)$ and $0\leq t_1<\cdots<t_n= T$. Set
$$
\|f\|:=\sup_{y\in (\mathbb R^d)^n}|f(y)|,\quad \|f\|_{lip}:=\sup_{y,z\in(\mathbb R^d)^n,y\not=z}\frac{|f(y)-f(z)|}{|y-z|}.
$$
We want to show that there exists a constant $C(\|f\|,\|f\|_{lip})\in (0,\infty)$ that is only dependent on $\|f\|$ and $\|f\|_{lip}$, such that
\begin{equation}\label{upper-b-eq-0}
\begin{aligned}
&\mathbb E^G \left( e^{\Phi\left(B\right)} \right)
\leq
 \exp\bigg\{\sup_{\eta\in(M_{G}^{2,0}(0,T))^d, |\eta|\leq C(\|f|,\|f\|_{lip})}\mathbb E^G\bigg(\Phi\left(B^\eta\right) - H_T^G(\eta)\bigg)\bigg\}
\end{aligned}
\end{equation}

Set $\phi(x_1,\cdots,x_n)=e^{f(x_1,\cdots,x_n)}$. Then $e^{-\|f\|}\leq \|\phi\|\leq e^{\|f\|}$, and  there exists a constant $C_1(\|f\|,\|f\|_{lip})\in (0,\infty)$ that is only dependent on $\|f\|$ and $\|f\|_{lip}$, such that
$$
\|\phi\|_{lip}\leq e^{\|f\|}\sup_{y,z\in(\mathbb R^d)^n,y\not=z}\frac{ e^{|f(y)-f(z)|}-1}{|y-z|}\leq e^{\|f\|}\sup_{\lambda>0}\frac{ e^{\lambda\|f\|_{lip} }-1}{\lambda}\leq C_1(\|f\|,\|f\|_{lip})
$$
For $t_{n-1}\le
t\leq t_{n}$, set
$$
\begin{aligned}
v_{n}(t,x_1,\cdots, x_{n-1},x)=& \mathbb
E^G\left(\phi(x_1,\cdots,x_{n-1}, x+ B_{t_n}-B_t)\right ) \\
=&\mathbb
E^G\left(\phi(x_1,\cdots,x_{n-1}, x+\sqrt{t_n-t}B_{1})\right ),
\end{aligned}
$$
and for any $l=n-1,\cdots,1$, $t\in [t_{l-1},t_l]$, define
$$
\begin{aligned}
v_{l}(t,x_1,\cdots, x_{l-1},x)=& \mathbb
E^G\left(v_{l+1}(x_1,\cdots,x_{l-1}, x+B_{t_{l}}-B_t,x+B_{t_{l}}-B_t)\right )\\
=& \mathbb
E^G\left(v_{l+1}(x_1,\cdots,x_{l-1}, x+\sqrt{t_l-t}B_{1},x+\sqrt{t_l-t}B_{1})\right ).
\end{aligned}
$$
Then by the definition of $G$-Brownian motion,
$v_l:[t_{l-1},t_{l}]\times \mathbb R^d\to \mathbb
R,\,l=1,\cdots,n$ are the solutions of equations:
\begin{equation}\label{upper-b-eq-1}
\left\{\begin{aligned} &\frac{\partial}{\partial t}
v_l(t,x_1,\ldots,x_{l-1},x) +G(D^2_x
v_l(t,x_1,\ldots,x_{l-1},x)) =0,~~t\in [t_{l-1},t_{l}),\\
&v_l(t_{l},x_1,\ldots,x_{l-1},x)=v_{l+1}(t_{l},x_1,\ldots,x_{l-1},x,x), \\
&v_n(T,x_1,\ldots,x_{n-1},x)= \phi(x_1,\ldots,x_{n-1},x).
\end{aligned}
\right.
\end{equation}
Since $\phi$ is bounded, by the regularity result of Krylov (Theorem 6.4.3 in
\cite{krylov}), and Section 4 in Appendix C of \cite{peng-book-10},
there exists a constant $\alpha\in(0,1)$
only depending on $G$, $\underline{\sigma}$, $d$  and $\|\phi\|$ such that for each $l=1,\cdots,n$,  and for any
$\kappa\in (0,t_{l+1}-t_l)$,
$$
\sup_{(x_1,\cdots,x_{l})\in (\mathbb R^d)^{l}}\|v_l(\cdot,x_1,\ldots,x_{l-1},\cdot)\|_{
C^{1+\alpha/2,2+\alpha}([t_l,t_{l+1}-\kappa] \times \mathbb
R^d)}<\infty,
$$
where for given real function $u$ defined on $Q=[T_1,T_2]\times
\mathbb{R}^{d}$, and given constants $\alpha,\beta \in(0,1)$,
\begin{align*}
\left \Vert u\right \Vert _{C^{\alpha,\beta}(Q)}
=&\sup_{\substack{x,y\in
\mathbb{R}^{d},\ x\not =y\\s,t\in [T_1,T_2],s\not =t}}\frac{|u(s,x)-u(t,y)|}%
{|r-s|^{\alpha}+|x-y|^{\beta}},\\
\left \Vert u\right \Vert _{C^{1+\alpha,2+\beta}(Q)}   =& \left \Vert
u\right \Vert _{C^{\alpha,\beta}(Q)}+\left \Vert \partial_{t}u\right
\Vert _{C^{\alpha,\beta }(Q)}+\sum_{i=1}^{d}\left \Vert
\partial_{x^{i}}u\right \Vert _{C^{\alpha,\beta
}(Q)}\\
&+\sum_{i,j=1}^{d}\left \Vert \partial_{x^{i}x^{j}%
}u\right \Vert _{C^{\alpha,\beta}(Q)}.
\end{align*}
By the subadditivity of $\mathbb E^G$,  for all $1\leq l\leq n$, $(t,x_1,\cdots,x_{l-1})\in [t_{l-1},t_l]\times(\mathbb R^d)^{l-1}$,
$$
 \sup_{x,y\in\mathbb R^d,x\not=y}\frac{|v_l(t,x_1,\cdots,x_{l-1},x)-v_l(t,x_1,\cdots,x_{l-1},y)|}{|x-y|}\leq C_1(\|f\|,\|f\|_{lip}),
$$
and  for all $1\leq l\leq n$, $(x_1,\cdots,x_{l-1},x)\in (\mathbb R^d)^{l}$,
$$
|v_l(t,x_1,\cdots,x_{l-1},x)-v_l(s,x_1,\cdots,x_{l-1},x)|\leq \bar{\sigma} C_1(\|f\|,\|f\|_{lip})|t-s|^{1/2}.
$$
Therefore, $x\to\nabla_xv_l(t,x_1,\cdots,x_{l-1},x)$, $(t,x_1,\cdots,x_{l-1})\in [t_{l-1},t_l)\times(\mathbb R^d)^{l-1}$  are uniformly bounded.

 Set
$V_l(t,x_1,\cdots,x_{l-1},x)=\log v_l(t,x_1,\cdots,x_{l-1},x)$. Then
for $(t,x) \in [t_{l-1},t_{l}) \times \mathbb R^d,$
$$
\frac{\partial V_l}{\partial t}=-G\left(\frac{D_x^2v_l}{v_l}\right),
~~ \nabla_x V_l =\frac{ \nabla_xv_l }{v_l},
$$
and
$$
D^2_xV_l=-(\partial_{x^i}V_l \partial_{x^j}V_l)_{i,j=1}^d+\frac{D_x^2v_l}{v_l}.
$$
Therefore, $\mathbb R^d\ni x\to  \nabla_xV_l(t,x_1,\cdots,x_{l-1},x)$, $(t,x_1,\cdots,x_{l-1})\in [t_{l-1},t_l)\times(\mathbb R^d)^{l-1}$  are uniformly bounded, i.e.,there exists a constant $C_2(\|f\|,\|f\|_{lip})\in (0,\infty)$ that is only dependent on $\|f\|$ and $\|f\|_{lip}$, such that  for all $1\leq l\leq n$, $(t,x_1,\cdots,x_{l-1},x)\in [t_{l-1},t_l)\times(\mathbb R^d)^{l}$,
$$
| \nabla_xV_l(t,x_1,\cdots,x_{l-1},x)| \leq C_2(\|f\|,\|f\|_{lip})
$$
and for any $\kappa\in (0,t_l-t_{l-1})$, $\mathbb R^d\ni x\to  \nabla_xV_l(t,x_1,\cdots,x_{l-1},x)$, $(t,x_1,\cdots,x_{l-1})\in [t_{l-1},t_l-\kappa)\times(\mathbb R^d)^{l-1}$  are uniformly  Lipschitz continuous.
Define
$$
U_l(t,x_1,\cdots,x_{l-1},x)=
\nabla_x V_l(t,x_1,\cdots,x_{l-1},x)I_{[t_{l-1},t_l)}(t).
$$

By the Picad iterative approach (cf. \cite{Gao-spa-09}), for any $\kappa\in (0,t_1)$,    the stochastic differential equation
$$
\left\{\begin{aligned} dX_t^{(1)}=&U_1(t,X_{t}^{(1)})d\langle
B\rangle_t+dB_t,~~t\in[0,t_1-\kappa],\\
X_0^{(1)}=&0,
\end{aligned}
\right.
$$
has a unique continuous solution  $\{X_t^{(1)},t\in
[0,t_1-\kappa]$. From the arbitrariness of $\kappa$, and noting that
for any $p\geq 1$,
$$
\mathbb E^G\left(|X_t^{(1)}-X_s^{(1)}|^p\right)\leq 2^p\left(\|U_1\|^p\bar{\sigma}^p|t-s|^p+\bar{\sigma}^{p/2}|t-s|^{p/2}\right)
$$
there exists a unique   continuous process  $\{X_t^{(1)},t\in
[0,t_1)$ such that
$$
\left\{\begin{aligned} dX_t^{(1)}=&U_1(t,X_{t}^{(1)})d\langle
B\rangle_t+dB_t,~~t\in[0,t_1],\\
X_0^{(1)}=&0,
\end{aligned}
\right.
$$
Recursively, for any $1\leq l\leq n$,   there exists a unique   continuous process   $\{X_t^{(1)},t\in
[t_{l-1},t_l] \}$, such that
\begin{equation}\label{upper-b-eq-2}
\left\{\begin{aligned} dX_t^{(l)}=&U_l(t,X_{t_1}^{(1)},\cdots,
X_{t_{l-1}}^{(l-1)},X_t^{(l)})d\langle
B\rangle_t+dB_t,~~t\in[t_{l-1},t_l],\\
X_{t_{l-1}}^{(l)}=&X_{t_{l-1}}^{(l-1)}.
\end{aligned}
\right.
\end{equation}
Define
$$
\quad \tilde{\eta}_t=U_l(t,X_{t_1}^{(1)},\cdots,
X_{t_{l-1}}^{(l-1)},X_t^{(l)})~, t\in [t_{l-1},t_l), ~l=1,\cdots,n; \quad  \tilde{\eta}_T=0
$$
and for any $l=1,\cdots,n$, and $t\in[t_{l-1},t_l)$,
$$
\begin{aligned}
K_t^{(l)}=&\int_{t_{l-1}}^t\bigg(G\left(\frac{D_x^2v_l}{v_l}\right)(t,X_{t_1}^{(1)},\cdots,
X_{t_{l-1}}^{(l-1)},X_s^{(l)})ds\\
&\quad\quad\quad- \frac{1}{2}\frac{D_x^2v_l}{v_l}(t,X_{t_1}^{(1)},\cdots,
X_{t_{l-1}}^{(l-1)},X_s^{(l)}) d\langle B\rangle_s \bigg).
\end{aligned}
$$
Then, by Proposition 1.4 in \cite{peng-book-10},   $ \mathbb E^G(-K_t^{(l)})=0$ for any $t\in[t_{l-1},t_l)$,  and by  Lemma \ref{V-formula-lem-1}, for any $t\in[t_{l-1},t_l)$, $K_t^{(l)}(\omega)\geq 0$ and $t\to K_t^{(l)}$ is increasing for q.s. $\omega$.   Set $K_{t_l}^{(l)}=\lim_{t\uparrow t_l}K_t^{(l)}$.

By It\^o formula
for $G$-Brownian motion (cf. \cite{Gao-spa-09}), for any $
t\in[t_{l-1},t_l)$,
$$
\begin{aligned} &V_l(t,X_{t_1}^{(1)},\cdots,
X_{t_{l-1}}^{(l-1)},X_t^{(l)})-V_l(t_{l-1},X_{t_1}^{(1)},\cdots,
X_{t_{l-1}}^{(l-1)},X_{t_{l-1}}^{(l)})\\
=&\int_{t_{l-1}}^t\frac{\partial V_l}{\partial t}(t,X_{t_1}^{(1)},\cdots,
X_{t_{l-1}}^{(l-1)},X_s^{(l)})ds+\int_{t_{l-1}}^t\nabla_x V_l(t,X_{t_1}^{(1)},\cdots,
X_{t_{l-1}}^{(l-1)},X_s^{(l)})dX_s^{(l)}\\
&+\frac{1}{2}\int_{t_{l-1}}^t  D_x^2 V_l (t,X_{t_1}^{(1)},\cdots,
X_{t_{l-1}}^{(l-1)},X_s^{(l)}) d\langle B\rangle_s \\
=&-K_t^{(l)} +H_t^G(\tilde{\eta})-H_{t_{l-1}}^G(\tilde{\eta}) +\int_{t_{l-1}}^t\tilde{\eta}_sdB_s.
\end{aligned}
$$
Since $\mathbb R^d\ni x\to V_l(t,x_1,\cdots,x_{l-1},x)$, $(t,x_1,\cdots,x_{l-1})\in [t_{l-1},t_l)\times(\mathbb R^d)^{l-1}$  are uniformly  Lipschitz continuous, and $\mathbb [t_{l-1},t_l]\ni t\to V_l(t,x_1,\cdots,x_{l-1},x)$, $(x_1,\cdots,x_{l-1},x)\in (\mathbb R^d)^{l}$  are $1/2$-uniformly H\"older continuous, we have that
$$
\lim_{t\uparrow t_l}\mathbb E^G\left(|V_l(t,X_{t_1}^{(1)},\cdots,
X_{t_{l-1}}^{(l-1)},X_t^{(l)})-V_l(t,X_{t_1}^{(1)},\cdots,
X_{t_{l-1}}^{(l-1)},X_{t_l}^{(l)})|\right)=0.
$$
Hence, $K_{t_l}\in L_G^1(\Omega_T)$, and $\lim_{t\uparrow t_l}\mathbb E^G\left(|K_t^{(l)}-K_{t_l}^{(l)}|\right)=0$,   $\mathbb E^G\left( -K_{t_l}^{(l)} \right)=0$, and
$$
\begin{aligned}
&V_l(t,X_{t_1}^{(1)},\cdots,
X_{t_{l-1}}^{(l-1)},X_{t_l}^{(l)})-H_{t_l}^G(\tilde{\eta})\\
=& V_{l-1}(t,X_{t_1}^{(1)},\cdots,
X_{t_{l-1}}^{(l-1)})-H_{t_{l-1}}^G(\tilde{\eta})-K_{t_l}^{(l)}+\int_{t_{l-1}}^t\tilde{\eta}_sdB_s,
\end{aligned}
$$
which yields that
$$
\begin{aligned}
\Phi(B^{\tilde{\eta}})-H_{T}^G(\tilde{\eta})
=&V_n(t,X_{t_1}^{(n)},\cdots,X_{t_n}^{(n)})-H_{t_n}^G(\tilde{\eta})\\
=& V_{1}(0,0)-\sum_{l=1}^nK_{t_l}^{(l)}+\int_{0}^T\tilde{\eta}_sdB_s.
\end{aligned}
$$
Therefore,
$$
\Phi(B^{\tilde{\eta}})-H_{T}^G(\tilde{\eta})-\int_{0}^T\tilde{\eta}_sdB_s=V_1(0,0)-\sum_{l=1}^nK_{t_l}^{(l)}
$$
and
$$
\begin{aligned}
\mathbb E^G\left(\Phi(B^{\tilde{\eta}})-H_{T}^G(\tilde{\eta})\right)=& V_1(0,0).
\end{aligned}
$$
Since $\sum_{l=1}^nK_{t_l}^{(l)}\geq 0$, $q.s.$,
we obtain that
$$
\begin{aligned}
&\mathbb E^G \left( e^{\Phi\left(B\right)} \right)=\mathbb E^G\left(\exp\left\{\Phi\left(B^{\tilde{\eta }}\right)- \int_0^T
\tilde{\eta}_sdB_s-H_T^G(\tilde{\eta})\right\}\right) \\
=& e^{V_1(0,0)}\mathbb E^G\left(e^{-\sum_{l=1}^nK_{t_l}^{(l)}}\right)\\
\leq&
\exp\left\{\mathbb E^G\left(\Phi\left(B^{\tilde{\eta }}\right)-H_T^G(\tilde{\eta})\right)\right\}.
\end{aligned}
$$

Now, for $m,N\geq 2/\min_{1\leq i\leq n}(t_l-t_{l-1})$, define
$$
\eta_s^{(m,N)}=\sum_{l=1}^n\sum_{k=1}^N \tilde{\eta}_{t_{l-1}+(k-1)(t_l-t_{l-1}-1/m)/N}I_{[t_{l-1}+\frac{(k-1)(t_l-t_{l-1}-1/m)}{N}, t_{l-1}+\frac{k(t_l-t_{l-1}-1/m)}{N})}(s).
$$
Then $\eta_s^{(m,N)}\in M_{G}^{2,0}(0,T)$,   $|\eta_s^{(m,N)}|\leq C_2(\|f|,\|f\|_{lip})$, and
$$
\lim_{m\to\infty}\limsup_{N\to\infty}\mathbb E^G\left(\int_0^T|\eta_s^{(m,N)}-\tilde{\eta}_s|^2\right)=0.
$$
Therefore,
\begin{equation}\label{upper-b-eq-3}
\begin{aligned}
&\mathbb E^G\left(\Phi\left(B^{\tilde{\eta }}\right)-H_T^G(\tilde{\eta})\right)
\leq
 \sup_{\eta\in(M_{G}^{2,0}(0,T))^d, |\eta|\leq C_2(\|f|,\|f\|_{lip})}\mathbb E^G\bigg(\Phi\left(B^\eta\right) - H_T^G(\eta)\bigg) ,
\end{aligned}
\end{equation}
and so (\ref{upper-b-eq-0}) holds.

For general bounded function $\Phi\in L_G^1(\Omega_T)$, choose a sequence
$\{\Phi_n,n\geq 1\}\subset L_{ip}(\Omega_T)$ of uniformly bounded and Lipschitz
continuous functions  such that
$$
\lim_{n\to\infty}\mathbb E^G(|\Phi_n-\Phi|^2)=0.
$$
Then
$$
\lim_{n\to\infty}\mathbb E^G(|e^{\Phi_n}-e^\Phi|)=0.
$$
By the above proof, there exists a sequence of positive constants $C_n$ such that
$$
\log\mathbb E^G\left(\exp\{\Phi_n \}\right)
\leq
 \sup_{\eta\in(M_{G}^{2,0}(0,T))^d, |\eta|\leq C_n}\mathbb E^G\bigg(\Phi_n\left(B^\eta\right) - H_T^G(\eta)\bigg).
$$
Set $\mathbb D=\cup_{n\geq 1} \{\eta\in(M_{G}^{2,0}(0,T))^d; |\eta|\leq C_n\}$. Then
$$
\log\mathbb E^G\left(\exp\{\Phi_n \}\right)
\leq
 \sup_{\eta\in \mathbb D}\mathbb E^G\bigg(\Phi_n\left(B^\eta\right) - H_T^G(\eta)\bigg).
$$
Since
$$
\begin{aligned}
&
\bigg|\sup_{\eta\in \mathbb D}\mathbb E^G\left(\Phi_n\left(B^\eta \right)- H_T^G(\eta)\right)
-\sup_{\eta\in \mathbb D}\mathbb E^G\left(\Phi\left(B^\eta\right)- H_T^G(\eta)\right)\bigg|\\
\leq &\sup_{\eta\in \mathbb D}\mathbb E^G
\left(\left|\Phi_n\left(B^\eta \right)-\Phi\left(B^\eta\right)\right|\right),
\end{aligned}
$$
and the same  proof as (\ref{lower-b-eq-1}) yields that
$$
\begin{aligned}
&\lim_{n\to\infty}\sup_{\eta\in \mathbb D}\mathbb E^G
\left(\left|\Phi_n\left(B^\eta \right)-\Phi\left(B^\eta\right)\right|\right)=0,
\end{aligned}
$$
we obtain
\begin{equation}\label{upper-b-eq-4}
\begin{aligned}
&\mathbb E^G \left( e^{\Phi} \right)  \leq
\exp\bigg\{\sup_{\eta\in \mathbb D}E^G\left(\Phi\left(B^\eta\right) -H_T^G(\eta)\right)\bigg\}
\end{aligned}
\end{equation}
which with together (\ref{lower-b-eq-3}) yields the conclusion of
Theorem \ref{G-variation-representation-thm}.

\end{proof}

\section{An abstract large deviation principle for functionals of $G$-Brownian motion}

In this section, we apply the variation representation to study the
large deviations for functionals of $G$-Brownian motion. The inverse of the Varadhan Lemma under a
$G$-expectation is presented.  An abstract large deviation principle for functionals of $G$-Brownian motion is obtained.

Let  $(\mathcal Y,\rho)$ be a Polish space and let $\Psi^{\epsilon }:
\Omega_T \times \mathbb A\rightarrow   \mathcal Y$ be a map. Set
$$
Z^{\epsilon}:= \Psi^{\epsilon }(\sqrt{\epsilon}B, \langle B\rangle).
$$

Define
\begin{equation}\label{H-def}
{\mathbb H} =\left\{f(\cdot)=\int_0^\cdot f'(s)ds; f'\in L^2([0,T],\mathbb R^d)\right\},\quad
\|f \|_{H}=\|f'\|_{L^2};
\end{equation}
\begin{equation}\label{G-def}
{\mathbb G} =\left\{g(\cdot)=\int_0^\cdot g'(s)ds;  g'\in L^2([0,T],\mathbb R^{d\times d})\right\},\quad
\|g \|_{G}=\int_{0}^T\|g'(t)\|_{HS} dt;
\end{equation}
and
\begin{equation}\label{A-def}
\begin{aligned}
\mathbb A =\bigg\{g\in \mathbb G;\ ~ g'(s)\in \Sigma
\mbox{ for any } s\in[0,T]\bigg\}.
\end{aligned}
\end{equation}
Then $(\mathbb A,\|\cdot\|_{G})$ is a closed convex subset of $\mathbb G$. We also denote
$$
\mathbb H_s  =\{f\in\mathbb H; f'(t)=\theta_1I_{[0,t_1]}(t)+\sum_{i=2}^{m} \theta_i I_{(t_{i-1},t_i]}(t),0<t_1<\cdots<t_n=T,\theta_i\in\mathbb R^d\}
$$
and
$$
\|g\|:=\sup_{t\in[0,T]}\|g(t)\|_{HS},\quad g\in\mathbb A;\quad  \|f\|:=\sup_{t\in[0,T]}|f(t)|,\quad f\in\mathbb H,
$$
where $\|A\|_{HS}:=\sqrt{\sum_{ij} a_{ij}^2}$ is the Hilbert-Schmidt
norm of a matrix $A=(a_{ij})$.  Define
$$
\rho_{HG}((f_1,g_1),(f_2.g_2))=\|f_1-f_2\|+\|g_1-g_2\|,\quad (f_1,g_1),(f_2.g_2)\in\mathbb H\times\mathbb A.
$$

We introduce the following {\it Assumption (A)}:

(A0).\quad  For any $\Phi\in C_b(\mathcal Y)$, $\Phi(Z^{\epsilon})$
is quasi-continuous;

There exists a map $\Psi:
\mathbb H\times \mathbb A \rightarrow \mathcal Y$   such that the following
conditions (A1), (A2) and (A3) hold:

(A1).\quad  For each $N\geq 1$,  if $f_{n},n\geq 1$,  $f\in
\mathbb H$, $g_n\in \mathbb A$  and $g\in \mathbb A$ satisfy  that $\|f_n\|_H\leq N$, $\|f\|_H\leq N$,
  $\|f_{n}-f\|_H\to 0$ and $\|g_n-g\|_{G}\to 0$, then
$$
\Psi\left( f_{n},g_n\right) \rightarrow \Psi\left(f,g\right);
$$

(A2).\quad  For $\Phi\in C_b(\mathcal Y)$,     for each $r>0$,
\begin{equation*}
\begin{aligned}
\lim_{\epsilon\to
0}\sup_{\eta\in(M_{G}^{2}(0,T))^d\cap \mathcal B_b(\Omega_T)\atop\int_0^T\mathbb
|\eta_s|^2ds\leq r}  \mathbb{E}^G\bigg(\bigg|&\Phi \circ
\Psi^\epsilon\left(\sqrt{\epsilon}B_\cdot+\int_0^\cdot
\eta_sd\langle B \rangle_s,\langle B \rangle\right)\\
&-  \Phi \circ \Psi
\left(\int_0^\cdot \eta_sd\langle B \rangle_s, \langle B \rangle\right)  \bigg|\bigg) =
0;
\end{aligned}
\end{equation*}

(A3).\quad There exists a sequence of continuous maps $\Psi^{(N)}:
(\mathbb H\times\mathbb A,\rho_{HG}) \rightarrow (\mathcal Y,\rho)$   such that for each $l\in(0,\infty)$,
$$
\lim_{N\to\infty} \sup_{\|f\|_H\leq l,g \in\mathbb A}\rho(\Psi(f,g),\Psi^{(N)}(f,g))=0.
$$

\begin{rmk}
Assumption (A) is slightly different from the classical case(cf. \cite{BoueDupuis-AP-98}).
We have an additional condition  (A0).  In the classical case, the condition (A0) is always true.
\end{rmk}

Let $I:\mathcal Y\rightarrow
[0,\infty ]$ be defined  by
\begin{equation} \label{rate-function-def-eq-1}
I(y)=\inf_{(f,g)\in \mathbb H\times\mathbb A}\left\{
\frac{1}{2}\int_0^T (f'(s),g'(s)f'(s))ds;~y
=\Psi\left(\int_{0}^{\cdot }g'(s)f'(s)ds,g\right)\right\}.
\end{equation}

Since  $\varphi(t)=\int_{0}^{t }g'(s)f'(s)ds, t\in[0,T]$ yields that
$f'(s)=g'(s)^{-1}\varphi'(s)$, it is easy to get the following
representation of $I(y)$:
\begin{equation}\label{rate-function-def-eq-4}
I(y)=\inf_{(f,g)\in \mathbb H\times \mathbb A}\left\{J(f,g),~y=\Psi(f,g)\right\},~~~y\in\mathcal Y,
\end{equation}
where
\begin{equation}\label{rate-function-def-eq-2}
J(f,g)=\left\{\begin{array}{ll}\displaystyle\frac{1}{2} \int_0^T
(f'(s),(g'(s))^{-1}f'(s))ds,~&~(f,g)\in \mathbb H\times \mathbb A,\\
\\
+\infty,&~otherwise.
\end{array}
\right.
\end{equation}

\begin{lem}\label{quadratic-process-G-expectation-lem}
(1). \quad Let $\Upsilon:  (\mathbb A,\|\cdot\|)\to \mathbb R$ be a bounded continuous function. Then
\begin{equation}\label{quadratic-process-G-expectation-lem-eq-1}
\mathbb E^G\left(\Upsilon\left(\langle B\rangle \right)\right)=\sup_{g\in\mathbb A}\Upsilon\left(g \right).
\end{equation}

(2). \quad Let $\Phi\in C_b(\mathcal Y)$, and let $\Psi:
\mathbb H\times \mathbb A \rightarrow \mathcal Y$  satisfy (A1) and (A3). Then for each function $f\in \mathbb H$,
\begin{equation}\label{quadratic-process-G-expectation-lem-eq-2}
\begin{aligned}
&  \mathbb{E}^G\left(\Phi \circ
\Psi\left( \int_0^\cdot f'(s)d\langle B \rangle_s,\langle B \rangle\right) -H_T^G(f)\right)\\
=&
\sup_{ g \in  \mathbb A}
 \left(\Phi \circ
\Psi\left(\int_0^\cdot g'(s)f'(s)ds,g\right) -\frac{1}{2}\int_0^T (f'(s),g'(s)f'(s)) ds\right).
\end{aligned}
\end{equation}

\end{lem}

\begin{proof}
(1). Firstly, let us show (\ref{quadratic-process-G-expectation-lem-eq-1}) for
$$
\Upsilon(g) =\psi(g_{t_1},g_{t_2}-g_{t_1},\cdots,g_{t_m}-g_{t_{m-1}})
$$
where $\psi$ is bounded continuous in $(\mathbb R^{d\times d})^m$, and $0<t_1<t_2<\cdots<t_m\leq T$. Since $\langle B\rangle_t-\langle B\rangle_s$ is independent of $\Omega_s$ for any $s<t$, by Chapter III, Theorem 5.3 in \cite{peng-book-10}, we have that
$$
\begin{aligned}
&\mathbb E^G\left(\psi(\langle B\rangle_{t_1},\langle B\rangle_{t_2}-\langle B\rangle_{t_1},\cdots,\langle B\rangle_{t_m}-\langle B\rangle_{t_{m-1}})\right)\\
=&\sup_{\theta_1,\theta_2,\cdots,\theta_m\in \Sigma}\psi(\theta_1 t_1 ,\theta_2(t_2-t_1),\cdots, \theta_m(t_m-t_{m-1})).
\end{aligned}
$$
For any $\theta_1,\theta_2,\cdots,\theta_m\in \Sigma$, set $g'(t)=\theta_1I_{[0,t_1]}+\sum_{i=2}^m \theta_i I_{(t_{i-1},t_i]}$. Then  $\psi(\theta_1 t_1 ,\theta_2(t_2-t_1),\cdots, \theta_m(t_m-t_{m-1}))=\psi(g(t_1) ,g(t_2)-g(t_1),\cdots, g(t_m)-g(t_{m-1}))$. Therefore,
$$
\begin{aligned}
&\sup_{\theta_1,\theta_2,\cdots,\theta_m\in \Sigma}\psi(\theta_1 t_1 ,\theta_2(t_2-t_1),\cdots, \theta_m(t_m-t_{m-1}))\\
\leq &\sup_{g\in\mathbb A}\psi(g(t_1) ,g(t_2)-g(t_1),\cdots, g(t_m)-g(t_{m-1})).
\end{aligned}
$$
On the other hand, for any $g\in \mathbb A$, set $\theta_1=g(t_1)/t_1$, and $\theta_i=(g(t_i)-g(t_{i-1}))/(t_i-t_{i-1})$ for $i=2,\cdots,m$. Then $\theta_i\in \Sigma$  and
$g(t_i)-g(t_{i-1})=\theta_i (t_i-t_{i-1})$ for any $i=1,\cdots,m$, where $t_0=0$. Therefore,
$$
\begin{aligned}
&\sup_{\theta_1,\theta_2,\cdots,\theta_m\in \Sigma}\psi(\theta_1 t_1 ,\theta_2(t_2-t_1),\cdots, \theta_m(t_m-t_{m-1}))\\
\geq &\sup_{g\in\mathbb A}\psi(g(t_1) ,g(t_2)-g(t_1),\cdots, g(t_m)-g(t_{m-1})).
\end{aligned}
$$
Therefore, (\ref{quadratic-process-G-expectation-lem-eq-1}) holds in this case.

Next, we assume that  $\Upsilon$ is Lipschitz continuous with respect to the uniform topology, i.e., there exists a
constant $l>0$ such that
$$
|\Upsilon(g)-\Upsilon(f)|\leq l \sup_{t\in [0,T]}\|g(t)-f(t)\|_{HS}~ \mbox{ for all
}~g,f\in \mathbb A.
$$
For each $N\geq 1$,  set $t_i^N=\frac{iT}{N}$, $1\leq i\leq N$. For any $(x_1,x_2,\cdots,x_N)\in \Sigma ^N$, set
$$
x^{(N)}(t):=x_1 (t\wedge t_1) +\sum_{i=2}^N x_i(t\wedge {t_{i}}-t\wedge {t_{i-1}}),~t\in [0,T],
$$
and  $\tilde{\psi}(x_1,x_2,\cdots,x_N)=\Upsilon\left(x^{(N)}\right)$. We can extend continuously $\tilde{\psi}(x_1,x_2,\cdots,x_N)$  to $ (\mathbb R^{d\times d})^N$. Define
$$
{\psi}(x_1,x_2,\cdots,x_N):=\tilde{\psi}\left(\frac{x_1}{t_1},\frac{x_2}{t_2-t_1},\cdots,\frac{x_N}{t_N-t_{N-1}}\right),~~(x_1,x_2,\cdots,x_N)\in (\mathbb R^{d\times d})^N.
$$
For $g\in\mathbb A$, set
$$
g^{(N)}(t):=\frac{g(t_1)}{t_1}(t\wedge t_1) +\sum_{i=2}^N \frac{(g(t_i)-g(t_{i-1}))}{t_i-t_{i-1}}(t\wedge {t_{i}}-t\wedge {t_{i-1}}), ~t\in [0,T],
$$
and  define
$$
\Upsilon^{(N)}(g):=\Upsilon(g^{(N)})=\psi\left(g(t_1), g(t_2)-g(t_{1}),
\cdots, g(t_N)-g(t_{N-1})\right).
$$
Then
$$
\mathbb E^G\left(\Upsilon^{(N)}\left(\langle B\rangle \right)\right)=\sup_{g\in\mathbb A}\Upsilon^{(N)}\left(g \right).
$$
and
$$
|\Upsilon(g)-\Upsilon^{(N)}(g)|\leq 2 l\max_{1\leq i\leq N}\sup_{t\in [t_{i-1},t_i]}\|g(t)-g(t_{i-1})\|_{HS}\leq 2l\bar{\sigma}T/N.
$$
Therefore,
$$
\mathbb E^G\left(\Upsilon\left(\langle B\rangle \right)\right)=\lim_{N\to\infty}\mathbb E^G\left(\Upsilon^{(N)}\left(\langle B\rangle \right)\right)=\lim_{N\to\infty}\sup_{g\in\mathbb A}\Upsilon^{(N)}\left(g \right)=\sup_{g\in\mathbb A}\Upsilon\left(g \right).
$$

Now, by the proof of Lemma 3.1, Chapter VI in \cite{peng-book-10},  for general bounded continuous $\Upsilon$, we can choose a sequence of Lipschitz functions $\Upsilon_N$  such that $\Upsilon_N \uparrow \Upsilon$. Therefore
$$
\begin{aligned}
\mathbb E^G\left(\Upsilon\left(\langle B\rangle \right)\right)=&\sup_{\theta\in\mathcal{ A
}_{0,\infty}^{\Gamma}}\sup_{N\ge1}E_{P_\theta} \left(\Upsilon_N\left(\langle B\rangle \right)\right)\\
=&\sup_{N\ge1}\mathbb E^G\left(\Upsilon_N\left(\langle B\rangle \right)\right)=\sup_{N\ge1}\sup_{g\in\mathbb A}\Upsilon_N\left(g \right)=\sup_{g\in\mathbb A}\Upsilon\left(g \right).
\end{aligned}
$$

(2).  Choose a sequence of simple functions $f_N'=\theta_1^NI_{[0,t_1^N]}+\sum_{i=2}^{m_N} \theta_i^N I_{(t_{i-1}^N,t_i^N]}$ such that $\sup_{N\geq 1}\|f_N\|_H<\infty$ and $\int_0^T|f'(s)-f_N'(s)|^2 ds\to 0$ as $N\to\infty$. Then
$$
\sup_{g\in\mathbb A}\left|\int_0^T (f'(s),g'(s)f'(s)) ds-\int_0^T (f_N'(s),g'(s)f_N'(s)) ds\right|\to0,
$$
and
$$
\sup_{g\in\mathbb A}\sup_{t\in[0,T]}\left|\int_0^tg'(s)f'(s)ds-\int_0^tg'(s)f_N'(s) ds\right|\to 0.
$$
Let $\Psi^{(N)}:
(\mathbb H\times\mathbb A,\rho_{HG}) \rightarrow (\mathcal Y,\rho)$   such that for any $l\in (0,\infty)$,
$$
\lim_{N\to\infty} \sup_{\|\varphi\|_H\leq l,g\in \mathbb A}\rho(\Psi(\varphi,g),\Psi^{(N)}(\varphi,g))=0.
$$
Define
$$
\Phi^{(N)}(g)=\Phi \circ
\Psi_N\left(\int_0^\cdot g'(s)f_N'(s) ds,g\right) -\frac{1}{2}\int_0^T (f_N'(s),g'(s)f_N'(s)) ds
$$
Then
$$
\lim_{N\to\infty} \sup_{g\in \mathbb A}\left|\Phi^{(N)}(g))-\left(\Phi \circ
\Psi\left(\int_0^\cdot g'(s)f'(s)ds,g\right) -\frac{1}{2}\int_0^T (f'(s),g'(s)f'(s)) ds\right)\right|=0,
$$
and by (1),
$$
\mathbb E^G(\Phi^{(N)}(\langle B\rangle))=\sup_{g\in\mathbb A}\Phi^{(N)}(g),
$$
Therefore, (\ref{quadratic-process-G-expectation-lem-eq-2}) holds.
\end{proof}

\begin{lem}\label{laplace-principle-lem-2}
Let (A2) hold. Then  for any  $\Phi\in C_b(\mathcal Y)$ and each $N\geq 1$,
\begin{equation*}
\begin{aligned}
\lim_{\epsilon\to
0}\sup_{\eta\in(M_{G}^{2,0}(0,T))^d\cap \mathcal B_b(\Omega_T)\atop\int_0^T\mathbb
E^G(|\eta_s|^2)ds\leq N}  \mathbb{E}^G\bigg(\bigg|&\Phi \circ
\Psi^\epsilon\left(\sqrt{\epsilon}B_\cdot+\int_0^\cdot
\eta_sd\langle B \rangle_s,\langle B \rangle\right)\\
&-  \Phi \circ \Psi
\left(\int_0^\cdot \eta_sd\langle B \rangle_s, \langle B \rangle\right)  \bigg|\bigg) =
0;
\end{aligned}
\end{equation*}

\end{lem}
\begin{proof}
For $\eta\in(M_{G}^{2,0}(0,T))^d\cap \mathcal B_b(\Omega_T)$,
we can write $\eta_s=\sum_{k=1}^n\eta_{t_{k-1}}I_{[t_{k-1},t_k)}(s)$.  For $r\in(0,\infty)$ fixed,
for any $\delta>0$, let $\phi(x)\in
lip(\mathbb R)$ satisfy $0\leq \phi\leq 1$, $\phi(x)=1$ for all
$|x|\leq r$ and $\phi(x)=0$ for all $|x|\geq r+\delta$. Define
$$
\hat{\eta}_t=\eta_t\phi\left(\int_0^t|\eta_s|^2ds\right).
$$
Then
$$
\left|\int_0^T|\hat{\eta}_s|^2 ds\right|\leq r+\delta,\quad
\left\{ \int_0^T |\eta_s|^2 ds \leq r \right\}\subset\left\{
 \hat{\eta}_s   =  \eta_s  \mbox
{ for any } s\in [0,T] \right\}.
$$
Set
$$
\hat{\eta}_s^N:=\sum_{k=1}^n\sum_{j=1}^N\hat{\eta}_{t_{k-1}+(j-1)(t_k-t_{k-1})/N} I_{[t_{k-1}+\frac{(j-1)(t_k-t_{k-1})}{N}, t_{k-1}+\frac{j(t_k-t_{k-1})}{N})}(s).
$$
Then
$$
\mathbb E^G\left(\int_0^T\left|\hat{\eta}_s^N-\hat{\eta}_s\right|ds\right)\leq \frac{\|\phi\|_{lip}}{2}\mathbb E^G\left(\sum_{k=1}^n\sum_{j=1}^N  |{\eta}_{t_{k-1}}|^3\frac{(t_k-t_{k-1})^2}{N^2}\right)\to 0 \mbox{ as } N\to\infty.
$$
Therefore,  for any $t\in[0,T]$,
$$
\mathbb E^G\left( \left|\int_{0}^{t}\hat{\eta}_s^N d\langle B\rangle_s-\int_{0}^{t}\hat{\eta}_s d\langle B\rangle_s\right|\right)\to 0 \mbox{ as } N\to\infty.
$$
In particular,  on $ \left\{\int_0^T |\eta_s|^2ds\leq r \right\}$,
$$
\int_{0}^{\cdot }\eta_s d\langle B\rangle_s=\int_{0}^{\cdot }\hat{\eta}_s d\langle B\rangle_s, ~q.s.,
$$
and
$$
\Psi^{\epsilon }\left(\sqrt{\epsilon }B_\cdot
+\int_{0}^{\cdot }\hat{\eta}_s d\langle B\rangle_s,\langle B\rangle\right)=\Psi^{\epsilon }\left(\sqrt{\epsilon }B_\cdot
+\int_{0}^{\cdot }{\eta}_s d\langle B\rangle_s,\langle B\rangle\right), ~q.s.
$$
Therefore, for any  $\Phi\in C_b(\mathcal Y)$ and   each $N\geq 1$, for all $\eta\in(M_{G}^{2,0}(0,T))^d\cap \mathcal B_b(\Omega_T)$ with $\int_0^T\mathbb E^G(|\eta_s|^2)ds\leq N$,
\begin{equation*}
\begin{aligned}
&   \mathbb{E}^G\bigg(\bigg|\Phi \circ
\Psi^\epsilon\left(\sqrt{\epsilon}B_\cdot+\int_0^\cdot
\eta_sd\langle B \rangle_s,\langle B \rangle\right) -  \Phi \circ \Psi
\left(\int_0^\cdot \eta_sd\langle B \rangle_s, \langle B \rangle\right)  \bigg|\bigg)\\
\leq & \frac{2\|\Phi\|N}{r}+\overline{\mathbb{E}}^G\bigg(\bigg|\Phi \circ
\Psi^\epsilon\left(\sqrt{\epsilon}B_\cdot+\int_0^\cdot
\tilde{\eta}_sd\langle B \rangle_s,\langle B \rangle\right)\\
 &\quad\quad\quad\quad -  \Phi \circ \Psi
\left(\int_0^\cdot\tilde{ \eta}_sd\langle B \rangle_s, \langle B \rangle\right)  \bigg|I_{\{\int_0^T |\eta_s|^2ds\leq r \}}\bigg)\\
\leq &\frac{2\|\Phi\|N}{r}+\sup_{\eta\in(M_{G}^{2}(0,T))^d\cap \mathcal B_b(\Omega_T)\atop\int_0^T\mathbb
|\eta_s|^2ds\leq r+\delta}\mathbb{E}^G\bigg(\bigg|\Phi \circ
\Psi^\epsilon\left(\sqrt{\epsilon}B_\cdot+\int_0^\cdot
\eta_sd\langle B \rangle_s,\langle B \rangle\right)\\
 &\quad\quad\quad\quad -  \Phi \circ \Psi
\left(\int_0^\cdot \eta_sd\langle B \rangle_s, \langle B \rangle\right)  \bigg| \bigg).
\end{aligned}
\end{equation*}
First, letting $\epsilon\to 0$, then $r\to\infty$, by (A2), we obtain the conclusion of the lemma.
\end{proof}

\begin{thm}\label{laplace-principle}
Suppose that the assumption (A) holds. Then

(1). For any $L\in[0,\infty)$, $C_{L}:=\{y;I(y)\leq L)\}$ is compact
in $\mathcal Y$;

(2). For any $\Phi\in C_b(\mathcal Y)$,
\begin{equation} \label{laplace-principle-eq-1}
\lim_{\epsilon\to 0} \left|\epsilon \log \mathbb E^G\left(\exp\left\{\frac{\Phi(Z^{\epsilon})}
{\epsilon}\right\}\right)-\sup_{y\in\mathcal Y}\left\{\Phi(y)-I(y)\right\}\right|=0.
\end{equation}
\end{thm}
\begin{proof}
(1). First, we prove that
$
C_{L} = \cap_{n \geq 1}\Gamma_{L+1/n},
$
where
$$
\Gamma_{L+1/n}=\left\{\Psi \left(f,g\right):J(f,g)\leq
L+\frac{1}{n}, (f,g) \in\mathbb H\times\mathbb A  \right\}.
$$
In fact, for $y\in C_{L}$ given,  for each $n \geq 1$, choose $ f_n \in \mathbb H $, $g_n\in\mathbb A$
such that  $y  = \Psi \left(f_n,g_n\right)$ and $J(f_n,g_n)\leq
L+\frac{1}{n}$. Since $n \geq 1$ is arbitrary, we have $C_{L}
\subseteq \cap_{n \geq 1}\Gamma _{L+1/n}$. Conversely, suppose $y
\in \cap_{n \geq 1}\Gamma_{L+1/n}$. Then, for some
$f_n\in\mathbb H$, $g_n\in\mathbb A$  with $y = \Psi\left(f_n,g_n\right)$, we
have that $J(f_n,g_n)\leq L+1/n$. Therefore, $ I(y) \leq L+
\frac{1}{n}$. Letting $n \rightarrow \infty$,
we obtain   $I(y) \leq L$. Thus $y \in
C_{L}$, and in turn, $\cap_{n \geq 1} \Gamma_{L+1/n} \subseteq
C_{L}$ follows.

(2).
From Theorem \ref{G-variation-representation-thm}, we have
\begin{equation*}
\begin{aligned}
&\epsilon \log \mathbb{E}^G\left( \exp\left\{ \frac{1}{\epsilon}
\Phi(Z^{\epsilon})\right\}\right) \\
=& \sup_{\eta\in(M_{G}^{2,0}(0,T))^d\cap \mathcal B_b(\Omega_T)} \mathbb{E}^G\left(\Phi \circ
\Psi^\epsilon\left(\sqrt{\epsilon} B^\eta,\langle B \rangle\right)
- H_T^G(\sqrt{\epsilon}\eta)\right)\\
=& \sup_{\eta\in(M_{G}^{2,0}(0,T))^d\cap \mathcal B_b(\Omega_T)} \mathbb{E}^G\left(\Phi \circ
\Psi^\epsilon\left(\sqrt{\epsilon}B_\cdot+\int_0^\cdot \eta_sd\langle B \rangle_s,\langle B \rangle\right)
-H_T^G(\eta)\right)\\
=& \sup_{\eta\in(M_{G}^{2,0}(0,T))^d\cap \mathcal B_b(\Omega_T)\atop\int_0^T\mathbb E^G(|\eta_s|^2)ds\leq \frac{4\|\Phi\|}{\underline{\sigma}}}
\mathbb{E}^G\left(\Phi \circ
\Psi^\epsilon\left(\sqrt{\epsilon}B_\cdot+\int_0^\cdot \eta_sd\langle B \rangle_s,\langle B \rangle\right)
-H_T^G(\eta)\right)
\end{aligned}
\end{equation*}
where $\|\Phi\|=\sup_{y\in\mathcal Y}|\Phi(y)|$, and the last
equality is due to that   if $\int_0^T\mathbb E^G(|\eta_s|^2)ds>
\frac{4\|\Phi\|}{\underline{\sigma}}$, then
$$
\mathbb{E}^G\left(\Phi \circ
\Psi^\epsilon\left(\sqrt{\epsilon}B_\cdot+\int_0^\cdot \eta_sd\langle B \rangle_s,\langle B \rangle\right)
-H_T^G(\eta)\right)\leq -\|\Phi\|.
$$
Therefore, by (A2) and Lemma \ref{laplace-principle-lem-2}, as $\epsilon\to 0$,
\begin{equation} \label{laplace-principle-eq-3}
\begin{aligned}
\lim_{\epsilon\to 0}  \bigg|&\epsilon \log \mathbb E^G\left(\exp\left\{\frac{\Phi(Z^{\epsilon})}
{\epsilon}\right\}\right)\\
&- \sup_{\eta\in(M_{G}^{2,0}(0,T))^d\cap \mathcal B_b(\Omega_T)\atop\int_0^T\mathbb E^G(|\eta_s|^2)ds\leq \frac{4\|\Phi\|}{\underline{\sigma}}}
\mathbb{E}^G\left(\Phi \circ
\Psi\left(\int_0^\cdot \eta_sd\langle B \rangle_s,\langle B \rangle\right) -H_T^G(\eta)\right)\bigg| =0.
\end{aligned}
\end{equation}
Since for each $\eta\in(M_{G}^{2,0}(0,T))^d\cap \mathcal B_b(\Omega_T)$,
$$
\begin{aligned}
&      \Phi \circ
\Psi\left(\int_0^\cdot \eta_sd\langle B \rangle_s,\langle B \rangle\right) -H_T^G(\eta) \\
\leq &    \sup_{(f,g)\in\mathbb H\times\mathbb A}\left(\Phi \circ
\Psi\left(\int_0^\cdot g'(s)f'(s)ds,g\right) -\frac{1}{2}\int_0^T (f'(s),g'(s)f'(s)) ds\right) \\
=&   \sup_{y\in\mathcal Y}\sup_{(f,g)\in\mathbb H\times\mathbb
A,y=\Psi(\int_0^\cdot g'(s)f'(s)ds,g)}\left(\Phi(y) -\frac{1}{2}\int_0^T (f'(s),g'(s)f'(s)) ds\right) \\
= &\sup_{y\in\mathcal Y}\left\{\Phi(y)-I(y)\right\}, ~q.s.,
\end{aligned}
$$
we obtain the upper bound:
$$
\limsup_{\epsilon\to 0} \epsilon \log \mathbb E^G\left(\exp\left\{\frac{\Phi(Z^{\epsilon})}
{\epsilon}\right\}\right)\leq \sup_{y\in\mathcal Y}\left\{\Phi(y)-I(y)\right\}.
$$

From Theorem \ref{G-variation-representation-thm}, we also have that
$$
\begin{aligned}
&\epsilon \log \mathbb{E}^G\left( \exp\left\{ \frac{1}{\epsilon}
\Phi(Z^{\epsilon})\right\}\right) \\
\geq & \sup_{f\in \mathbb H_s;\|f\|_H\leq\frac{4\|\Phi\|}{\underline{\sigma}} }\mathbb{E}^G\left(\Phi \circ
\Psi^\epsilon\left(\sqrt{\epsilon}B_\cdot+\int_0^\cdot f_s'd\langle B \rangle_s,\langle B \rangle\right) -H_T^G(f)\right).
\end{aligned}
$$
Thus, by (A2), (A3)
and Lemma \ref{quadratic-process-G-expectation-lem},
$$
\begin{aligned}
&   \liminf_{\epsilon\to 0} \epsilon \log \mathbb E^G\left(\exp\left\{\frac{\Phi(Z^{\epsilon})}
{\epsilon}\right\}\right)\\
\geq &\sup_{f\in \mathbb H_s,\|f\|_H\leq\frac{4\|\Phi\|}{\underline{\sigma}} }\mathbb{E}^G\left(\Phi \circ
\Psi\left( \int_0^\cdot f_s'd\langle B \rangle_s,\langle B \rangle\right) -H_T^G(f)\right)\\
= & \sup_{f\in \mathbb H_s,\|f\|_H\leq\frac{4\|\Phi\|}{\underline{\sigma}} }
\sup_{ g \in  \mathbb A}
 \left(\Phi \circ
\Psi\left(\int_0^\cdot g'(s)f'(s)ds,g\right) -\frac{1}{2}\int_0^T (f'(s),g'(s)f'(s)) ds\right)\\
= & \sup_{\|f\|_H\leq\frac{4\|\Phi\|}{\underline{\sigma}} }
\sup_{ g \in  \mathbb A}
 \left(\Phi \circ
\Psi\left(\int_0^\cdot g'(s)f'(s)ds,g\right) -\frac{1}{2}\int_0^T (f'(s),g'(s)f'(s)) ds\right).
\end{aligned}
$$
Then, letting $N\to\infty$, we obtain the lower bound:
$$
\begin{aligned}
 \liminf_{\epsilon\to 0} \epsilon \log \mathbb E^G\left(\exp\left\{\frac{\Phi(Z^{\epsilon})}
{\epsilon}\right\}\right)\geq &\sup_{y\in\mathcal Y}\left\{\Phi(y)-I(y)\right\}.
\end{aligned}
$$
Therefore, (\ref{laplace-principle-eq-1}) is valid.

\end{proof}

\begin{thm}\label{ldp-thm}
Suppose that the assumption (A) holds.  Then $\{Z^{\epsilon }
,\epsilon>0\}$ satisfies the large deviation principle in $\mathcal
Y$  with the rate function $I (y)$,    i.e., for any closed subset $F\subset \mathcal Y$,
\begin{equation}\label{ldp-ub}
\limsup_{\epsilon\to0}\epsilon\log
c^G(Z^{\epsilon }\in F)\leq-\inf_{y\in F}I(y),
\end{equation}
and for any open subset $O\subset \mathcal Y$,
\begin{equation}\label{ldp-lb}
\liminf_{\epsilon\to0}\epsilon\log
c^G(Z^{\epsilon }\in O)\geq-\inf_{y\in O}I(y).
\end{equation}

\end{thm}

\begin{proof} This is a consequence of Theorem \ref{laplace-principle}.
Its proof is the same as probability measure case.   For given open set $O$, for any  $y\in O$, choose continuous
map $\Phi:
\mathcal Y\to [0, 1]$ such that  $\Phi(y)=1$ and for any $ z\in O^c$,
$\Phi(z)=0$.     For any $ m\geq 1$, set $\Phi_m(z): =m(\Phi(z)-1)$, $z\in\mathcal Y $.  Then
$$
\begin{aligned}
\mathbb E^G\left(\exp\left\{\frac{1}{\epsilon}\Phi_m(Z^{\epsilon})\right\}\right)
&\leq
e^{-\frac{m}{ \epsilon }}c^G(Z^{\epsilon }\in O^c)+c^G(Z^{\epsilon }\in O).
\end{aligned}
$$
Therefore,
$$
\begin{aligned}
& \max\left\{\liminf_{\epsilon\to
0} \epsilon \log c^G(Z^{\epsilon }\in O), -m\right\}\\
\geq & \liminf_{\epsilon\to 0} \epsilon \log \mathbb E^G\left(\exp\left\{\frac{1}{\epsilon}
\Phi_m(Z^{\epsilon})\right\}\right)=\sup_{z\in
\mathcal Y }\left\{\Phi(z)-I(z)\right\}\geq -I(y).
\end{aligned}
$$
Letting $m\to +\infty$,    we obtain the lower bound.

Next, let us show the upper bound.  For closed set $F$ given,   for any $
y\not\in F$, choose continuous map $\Phi_y:  E\to[0, 1]$ such that
$\Phi_y(y)=1$ and for any $ z\in F$,   $\Phi_y(z)=0$. For any finite
set $A\subset F^c$, set $\Phi_A(z)=\max_{y\in A} \Phi_y(z)$. Then
$$
\begin{aligned}
\limsup_{\epsilon\to 0} \epsilon \log
c^G(Z^{\epsilon }\in F) &\leq \inf_{A\subset F^c\atop finite}\lim_{\epsilon\to
0} \epsilon \log
\mathbb E^G\left(\exp\left\{\frac{-m\Phi_A(Z^{\epsilon })}{\epsilon} \right\} \right)\\
&=-\sup_{A\subset F^c \atop finite }\inf_{z\in \mathcal
Y}\left\{m\Phi_A(z)+I(z)\right\}.
\end{aligned}
$$
Without loss of generality, we assume that $ l:=\sup_{A\subset F^c\atop finite
}\inf_{z\in \mathcal Y}J_A(z)<\infty$,  where $J_A(z)=
m\Phi_A(z)+I(z)$. Then $\{z;J_A(z)\leq l\}$ is
nonempty compact set for any finite $A$. Therefore, $\cap_{A\subset
F^c~finite }\{z;J_A(z)\leq l\}$ is nonempty, and so
$$
\begin{aligned}
l\geq &\inf_{z\in
\mathcal Y}\sup_{A\subset F^c\atop finite
}J_A(z)=\min \left\{m+\inf_{z\in
F^c}I(z),   \inf_{z\in
F}I(z)\right\}\stackrel{m\to\infty}{\longrightarrow }\inf_{z\in
F}I(z),
\end{aligned}
$$
which yields (\ref{ldp-ub}).

\end{proof}

\section{Large deviations for stochastic  flows driven by $G$-Brownian motion}

The
homeomorphic property with respect to initial values of the solution
for stochastic differential equations driven by $G$-Brownian motion
was obtained in \cite{Gao-spa-09} and the large deviations for
solutions $\{X^{\epsilon}(x,t),t\in[0,T],\}\subset C([0,T],\mathbb
R^p)$ of small perturbation stochastic differential equations
starting from $x$ (fixed) by $G$-Brownian motion were studied in
\cite{GaoJiang} by exponential estimates and
discretization/approximation techniques. In this section, we
consider large deviations for the flows
$\{X^{\epsilon}(x,t),(x,t)\in\mathbb R^p\times [0,T]\}\subset
C(\mathbb R^p\times [0,T],\mathbb R^p)$. The quasi continuity of the
flows is proved. A Kolmogorov criterion on weak convergence under
$G$-expectations is given. A large deviation principle for the
flows is established under the  Lipschitz condition. In the classical
framework, Large deviations for stochastic flows have been studied
extensively (see \cite{BaldiSanz}, \cite{BenArousCastell},
\cite{BDM-Bernoulli-10}, \cite{Gao-Ren}, \cite{MilletNualartSanz},
\cite{RenZhang} and references therein). For general theory of large
deviations and random perturbations, we refer to
\cite{DemboZeitouni}, \cite{DupuisEllis}, and
\cite{FreidlinWentzell}.

  For positive number $p\geq 1$ given, for each $N\geq 1$,  $\psi \in C(\mathbb R^p\times [0,T],\mathbb R^p)$, set
$$
\|\psi\|_{N}=  \sup_{x\in[-N, N]^p,t\in[0,T]}|\psi(x,t)|,
$$
and define
$$
\rho(\psi_1,\psi_2)=\sum_{N=1}^\infty
\frac{1}{2^N}\min\{\|\psi_1-\psi_2\|_{N},1\}, ~~\psi_1,\psi_2\in
C(\mathbb R^p\times [0,T],\mathbb R^p).
$$
Then  $(C(\mathbb R^p\times [0,T],\mathbb R^p),\rho)$ is a separable metric space.

Consider the following small perturbation stochastic differential
equation driven by a $d$-dimensional $G$-Brownian motion $B$:
\begin{equation}\label{SDE-G-pertubation}
X^{\epsilon}(x,t)=x+\int_{0}^{t}b^{\epsilon}(X^{\epsilon}(x,s))ds
+\sqrt{\epsilon}\int_{0}^{t}\sigma^{\epsilon}(X^{\epsilon}(x,s))dB_{s}+\int_{0}^{t}h^\epsilon (X^{\epsilon}(x,s))d\left \langle
B\right \rangle _{s},
\end{equation}
where
$$
b^{\epsilon}: \mathbb
R^p\rightarrow \mathbb R^p;\quad
\sigma^{\epsilon}=(\sigma^{\epsilon}_{i,j})_{1\leq i\leq p,1\leq j\leq d}: \mathbb R^p\rightarrow
\mathbb R^p\otimes\mathbb{R}^d,
$$
and
$$
 h^\epsilon=(h^{\epsilon,k})_{1\leq k\leq p}=((h_{ij}^{\epsilon,k})_{1\leq i,j\leq d})_{1\leq k\leq p}:\mathbb R^p\mapsto
(\mathbb{R}^{d\times d})^p, ~\epsilon\geq 0
$$
satisfy the following conditions:
\begin{flushleft}
$(H1)$.  $b^\epsilon $,  $ \sigma^\epsilon $ and $h^\epsilon$, $\epsilon\geq 0$ are
uniformly Lipschitz continuous, i.e.,
 there exists a
constant $L>0$ such that  for any $x, y\in \mathbb R^p$,
$$
\max\left\{|b^\epsilon  (x)-b^\epsilon  (y)|,\|\sigma^\epsilon  (x)-\sigma ^\epsilon (y)\|_{HS}, \max_{1\leq k\leq p} \|h^{\epsilon,k} (x)-h^{\epsilon,k} (y)\|_{HS}\right\}\leq L|x-y|.
$$

\end{flushleft}

\begin{flushleft}
$(H2)$.  $b^{\epsilon}$, $\sigma^{\epsilon}$ and $ h^{\epsilon}$   converge uniformly to $b:=b^0$, $\sigma: =\sigma^0$ and $h:=h^0$
respectively, i.e.,
$$
\lim_{\epsilon\to0} \sup_{x\in \mathbb R^p}\max\left\{
|b^{\epsilon}(x)-b(x)|,\|\sigma^{\epsilon}(x)-\sigma(x)\|_{HS},\max_{1\leq k\leq p} \|h^{\epsilon,k} (x)-h^k(x)\|_{HS}
\right\} =0.
$$
\end{flushleft}

Then by Theorem 4.1 in \cite{Gao-spa-09} and the Kolmogorov
criterion under $G$-expectation (cf. Theorem 1.36, Chapter VI in
\cite{peng-book-10})),  the  SDE (\ref{SDE-G-pertubation}) has a
unique solution $X^\epsilon=\{X^{\epsilon}(x,t),x\in\mathbb R^p,t\in
[0,T]\}\subset  C(\mathbb R^p\times [0,T],\mathbb R^p)$ and
$X^\epsilon(x,t)\in L^2_G(\Omega_T)$ for all $(x,t)\in \mathbb
R^p\times[0,T]$. Furthermore,  there exists a map $\Psi^\epsilon:
\Omega_T\times \mathbb A\to C\left(\mathbb R^p\times [0,T], \mathbb R^p\right)$
such that
\begin{equation}\label{Psi-epsilon-def}
\Psi^\epsilon(\sqrt{\epsilon} B, \langle B\rangle)=X^\epsilon.
\end{equation}

For any $(f,g)\in {\mathbb H}\times \mathbb A$,  let $ \Psi(f,g)(x,t)\in C(\mathbb
R^p\times [0,T], \mathbb R^p)$ be a unique solution of the
following ordinary differential equation:
\begin{equation}\label{Psi-def}
\begin{aligned}
\Psi(f,g)(x,t)=&x+\int_{0}^{t}b(\Psi(f,g)(x,s))ds
+\int_{0}^{t}\sigma(\Psi(f,g)(x,s))f'(s)ds\\
&+\int_{0}^{t}h(\Psi(f,g)(x,s)) dg(s).
\end{aligned}
\end{equation}

\begin{thm}\label{ldp-sde-main}
Let  $(H1) $  and $(H2)$ hold. Let
$X^\epsilon=\{X^{\epsilon}(x,t),x\in\mathbb R^p,t\in [0,T]\}$ be a
unique solution of the SDE (\ref{SDE-G-pertubation}). Then

(1). For any $\Phi\in C_b( C(\mathbb R^p\times [0,T], \mathbb
R^p))$,
\begin{equation} \label{ldp-sde-main-eq-1}
\lim_{\epsilon\to 0}  \left|\epsilon \log \mathbb
E^G\left(\exp\left\{\frac{\Phi(X^{\epsilon})}{\epsilon}\right\}\right)-\sup_{\psi\in
C(\mathbb R^p\times [0,T], \mathbb
R^p)}\left\{\Phi(\psi)-I(\psi)\right\}\right|=0 ,
\end{equation}
where
\begin{equation}\label{rate-function-G-SDE}
I(\psi)=\inf_{(f,g)\in\mathbb
H\times \mathbb A}\left\{J(f,g),~\psi=\Psi(f,g)\right\},~~~\psi\in C(\mathbb R^p\times
[0,T], \mathbb R^p).
\end{equation}

(2). For any closed subset $F$ in $ (C\left(\mathbb R^p\times [0,T], \mathbb R^p\right),\rho )$,
\begin{equation}
\limsup_{\epsilon\to0}\epsilon\log c^G\left(X^\epsilon\in
F\right)\leq-\inf_{\psi\in F}I(\psi)
\end{equation}
and for any open subset $O$ in $ (C\left(\mathbb R^p\times [0,T],
\mathbb R^p\right),\rho )$,
\begin{equation}
\liminf_{\epsilon\to0}\epsilon\log c^G\left(X^\epsilon\in
O\right)\geq-\inf_{\psi\in O}I(\psi),
\end{equation}

\end{thm}

\begin{proof}
By Theorem \ref{ldp-thm}, we only need to verify the conditions
(A0), (A1), (A2) and (A3)  for $\mathcal Y=C(\mathbb R^p\times [0,T],
\mathbb R^p)$, $Z^\epsilon=X^\epsilon$ and $\Psi$ defined by
(\ref{Psi-def}). These will be given in Lemma
\ref{A0-condition-proof}, Lemma
\ref{A1-A2-condition-proof} and Lemma \ref{A3-condition-proof}.
\end{proof}

\begin{rmk}
In particular, Theorem \ref{ldp-sde-main} yields that
$\{\{\sqrt{\epsilon}B_t, t\in [0,T]\},\epsilon>0\}$ satisfies a
large deviation principle,  which was
first obtained in \cite{GaoJiang} by the subadditive method.

\end{rmk}

\begin{lem}\label{A0-condition-proof}
Assume that $(H1) $ and $(H2)$ hold. Let $X=\{X(x,t),x\in\mathbb R^p,t\in [0,T]\}$ be a unique solution
of the SDE:
\begin{equation}\label{A0-condition-proof-eq-1}
X(x,t)=x+\int_{0}^{t}b(X(x,s))ds
+\int_{0}^{t}\sigma(X(x,s))dB_{s}+\int_{0}^{t}h (X(x,s))d\left \langle
B\right \rangle _{s}.
\end{equation}
Then for any
$\Phi\in C_{b}(C(\mathbb R^p\times [0,T], \mathbb R^p))$, $\Phi
(X)$ is quasi-continuous.

\end{lem}

\begin{proof}

First, we assume that $\Phi\in C_{b}(C(\mathbb R^p\times [0, T],
\mathbb R^p))$ is Lipschitz continuous, i.e., there exists a
constant $l>0$ such that
$$
|\Phi(\psi)-\Phi(\varphi)|\leq l \rho(\psi,\varphi)~ \mbox{ for all
}~\psi,\varphi\in C(\mathbb R^p\times [0, T], \mathbb R^p).
$$
For any $N\geq 1$,  for each $\psi\in C(\mathbb R^p\times [0, T],
\mathbb R^p)$, set $\psi^{(N)}(x,t)=\psi((-N)\vee x\wedge N, t)$, where $(-N)\vee x\wedge N=((-N)\vee x_1\wedge N,\cdots,(-N)\vee x_p\wedge N)$.
For given $N\geq 1$, for each $\psi=(\psi_1,\cdots,\psi_p)\in C(\mathbb R^p\times [0, T], \mathbb R^p)$, for any  $(x_1,\cdots,x_p,x_{p+1})\in [0,1]^{p+1}$, set
$$
\tilde{\psi}_j(x_1,\cdots,x_p,x_{p+1})=\psi_j^{(N)}\left(N(2x_1-1), \cdots, N(2x_p-1),Tx_{p+1}\right).
$$
For any $m\geq 1$,  the Bernstein polynomial of  $\tilde{\psi}_j$  is defined by
$$
B_m(\tilde{\psi}_j)(x_1,\cdots,x_p,x_{p+1})=\sum_{1\leq i_1,\cdots,i_{p+1}\leq m}\tilde{\psi}_j\left(\frac{i_1}{m},\cdots,\frac{i_p}{m},\frac{i_{p+1}}{m}\right)
\prod_{k=1}^{p+1}
\begin{pmatrix}
m \\
i_k \\
\end{pmatrix}
x_k^{i_k}(1-x_k)^{m-i_k} .
$$
Then, by Bernstein's theorem (cf. Theorem 3.1 and its proof in \cite{Kowalski}), we have that
$$
\begin{aligned}
&\left|\tilde{\psi}_j(x_1,\cdots,x_p,x_{p+1})-B_m(\tilde{\psi}_j)(x_1,\cdots,x_p,x_{p+1})\right|\\
\leq &\sup_{\sum_{k=1}^{p+1}|x_k-y_k|^2\leq 1/m}\left|\tilde{\psi}_j(x_1,\cdots,x_p,x_{p+1})-\tilde{\psi}_j(y_1,\cdots,y_p,y_{p+1})\right|\\
&+ \frac{p+1}{2m}\sup_{(x_1,\cdots,x_p,x_{p+1})\in[0,1]^{p+1}}\left|\tilde{\psi}_j(x_1,\cdots,x_p,x_{p+1})\right|.
\end{aligned}
$$
Since
$X(x,t)\in L^2_G(\Omega_T)$ for all $(x,t)\in \mathbb
R^p\times[0,T]$, we have  that
$$
\tilde{X}_j\left(\frac{i_1}{m},\cdots,\frac{i_p}{m},\frac{i_{p+1}}{m}\right) \in
L^2_G(\Omega_T), j=1,\cdots,p, 1\leq i_1,\cdots,i_{p+1}\leq m;
$$
and
$$
B_m(\tilde{X}_j)(x_1,\cdots,x_p,x_{p+1})\in
L^2_G(\Omega_T) \mbox{ for all } (x_1,\cdots,x_p,x_{p+1})\in [0,1]^{p+1}.
$$
For $x\in\mathbb R^p,t\in[0,T]$, Set
$$
\begin{aligned}
& X^{N,m}(x,t)\\
=&\left(B_m(\tilde{X}_1)\left((-1)\vee\frac{x+N}{2N}\wedge 1,\frac{t}{T}\right),\cdots,
B_m(\tilde{X}_p)\left((-1)\vee\frac{x+N}{2N}\wedge 1,\frac{t}{T}\right)\right).
\end{aligned}
$$
Noting that $\Phi\left(X^{N,m}\right)$ is a continuous function of $\tilde{X}_j\left(\frac{i_1}{m},\cdots,\frac{i_p}{m},\frac{i_{p+1}}{m}\right) $,  $j=1,\cdots,p$, $1\leq i_1,\cdots,i_{p+1}\leq m$, we obtain $\Phi\left(X^{N,m}\right) \in L^2_G(\Omega_T)$.

By Theorem 4.1 in \cite{Gao-spa-09},  for any $q\geq 2$,
\begin{equation}\label{A0-condition-proof-eq-2}
 {\mathbb E}^G(|X(x,t)-X(y,s)|^q)\leq C_{q,
T}(|x-y|^q+|s-t|^{q/2}).
\end{equation}
This yields by the Kolmogorov criterion under $G$-expectation (cf. Theorem 1.36, Chapter VI in \cite{peng-book-10})) that
for each $1\leq j\leq p$,
$$
\lim_{m\to\infty}\overline{\mathbb E
}^G\left(\sup_{\sum_{k=1}^{p+1}|x_k-y_k|^2\leq 1/m}\left|\tilde{X}_j(x_1,\cdots,x_p,x_{p+1})-\tilde{X}_j(y_1,\cdots,y_p,y_{p+1})\right|^2\right)=0,
$$
and
$$
\overline{\mathbb E}^G\left(\sup_{(x_1,\cdots,x_p,x_{p+1})\in[0,1]^{p+1}}\left|\tilde{X}_j(x_1,\cdots,x_p,x_{p+1})\right|^2\right)<\infty.
$$
Therefore,
$$
\lim_{m\to\infty}\overline{\mathbb E
}^G\left(\sup_{x\in[-N,N]^p,t\in[0,T]}\left| {X}(x,t)- {X}^{N,m}(x,t)\right|^2\right)=0,
$$
and by
\begin{align*}
\left|\Phi(X)-\Phi(X^{N,m})\right| &  \leq l \sup
_{x\in[-N,N]^p,t\in
[0,T]}|X(x,t)-X^{N,m}(x,t)|+\frac{l}{2^{N-1}},
\end{align*}
we obtain
$$
\lim_{N\to\infty}\lim_{m\to\infty}\overline{\mathbb E
}^G\left(\left|\Phi(X)-\Phi(X^{N,m})\right| ^2\right)=0,
$$
which implies that $\Phi (X)\in L^1_G(\Omega_T)$, and so
$\Phi (X)$ is quasi-continuous.

For general $\Phi\in C_b(C(\mathbb R^p\times [0, T], \mathbb R^p))$,
set $M=\sup_{\psi\in C(\mathbb R^p\times [0, T], \mathbb R^p)}
|\Phi(\psi)|$. For any $N\geq 1$,   set
$$
\Phi^{(N)}(\psi)=\inf_{\varphi\in C(\mathbb R^p\times [0, T],
\mathbb R^p)}\{\Phi(\varphi)+N\|\psi-\varphi\|_N\},~~\psi\in
C(\mathbb R^p\times [0,T], \mathbb R^p).
$$
Then (cf. Lemma 3.1, Chapter VI in \cite{peng-book-10}),
$|\Phi^{(N)}|\leq M$,
$$
|\Phi^{(N)}(\psi)-\Phi^{(N)}(\varphi)|\leq
N\|\psi-\varphi\|_N,~~\psi,\varphi\in C(\mathbb R^p\times [0,T],
\mathbb R^p),
$$
and for any $\psi\in C(\mathbb R^p\times [0,T], \mathbb R^p)$,
$\Phi^{(N)}(\psi)\uparrow \Phi(\psi)$ as $N\to\infty$. Therefore,
$\Phi^{(N)}(X)$ is quasi-continuous for all $N\geq 1$. For any
$\delta>0$, choose a compact subset $K\subset \Omega_T$ such that
$c^G(K^c)<\delta$ and for all $N\geq 1$, $\Phi^{(N)}(X)$ is
continuous on $K$. By Dini's Theorem, $\Phi^{(N)}(X)$ converges
uniformly to $\Phi(X)$ on $K$, and so $\Phi(X)$ is continuous on
$K$. Thus, $\Phi (X)$ is quasi-continuous.

\end{proof}

\begin{lem}\label{A1-A2-condition-proof}
 Assume that $(H1) $ and $(H2)$ hold.

(1).   For each $N\geq 1$,  if $f_{n},n\geq 1$,  $f\in
\mathbb H$, $g_n\in \mathbb A$  and $g\in \mathbb A$ satisfy  that $\|f_n\|_H\leq N$, $\|f\|_H\leq N$,
  $\|f_{n}-f\|_H\to 0$ and $\|g_n-g\|_{G}\to 0$, then
$$
\Psi\left( f_{n},g_n\right) \rightarrow \Psi\left(f,g\right).
$$

(2). For $\Phi\in C_b(C(\mathbb R^p\times [0,T], \mathbb R^p))$,
for each $N\geq 1$,

\begin{equation}\label{A1-A2-condition-proof-eq-1}
\begin{aligned}
 \lim_{\epsilon\to
0}\sup_{\eta\in(M_{G}^{2}(0,T))^d\cap \mathcal B_b(\Omega_T)\atop\int_0^T\mathbb
|\eta_s|^2ds\leq N}  \mathbb{E}^G\bigg(\bigg|&\Phi \circ
\Psi^\epsilon\left(\sqrt{\epsilon}B_\cdot+\int_0^\cdot \eta_sd\langle B
\rangle_s,\langle B
\rangle\right)\\
&-  \Phi \circ \Psi \left(\int_0^\cdot \eta_sd\langle B
\rangle_s,\langle B
\rangle\right)  \bigg|\bigg) = 0.
\end{aligned}
\end{equation}

\end{lem}

\begin{proof} (1). For any $m\geq 1$, set
$$
M_m=\sup_{|x|\leq m,t\in [0,T]}(\|\sigma(\Psi(f,g)(x,s))\|_{HS}+\max_{1\leq k\leq p}\|h^k(\Psi(f,g)(x,s))\|_{HS}\|).
$$
Then, there exists a constant $M\in(0,\infty)$ such that, on $\{ |x|\leq m,t\in [0,T]\}$,
$$
\begin{aligned}
&|\Psi(f,g)(x,t)-\Psi(f_n,g_n)(x,t)|\\
\leq &M\int_{0}^{t} |\Psi(f,g)(x,s)-\Psi(f_n,g)(x,s)|(1+|f'_n(s)|+\|g_n'(s)\|_{HS})ds\\
&
+ M_m\int_{0}^{t}(|f_n'(s)-f'(s)|+\|g_n'(s)-g'(s)\|_{HS})ds.
\end{aligned}
$$
By Gronwall's inequality,
$$
\begin{aligned}
&\sup_{|x|\leq m,t\in [0,T]}|\Psi(f,g)(x,t)-\Psi(f_n,g)(x,t)|\\
\leq & M_m\int_{0}^{T}(|f_n'(s)-f'(s)|+\|g_n'(s)-g'(s)\|_{HS})d se^{M\int_0^T(1+|f'_n(t)|+\|g_n'(t)\|_{HS})dt}\\
\leq & M_m\left(\sqrt{T}\|f_n-f\|_H+\|g_n-g\|_{G}\right) e^{M (T+\sqrt{NT}+p\bar{\sigma}T)}\to 0 \mbox{ as } n\to\infty.
\end{aligned}
$$

(2).
For any $ \eta\in(M_{G}^{2}(0,T))^d$ with
$\int_0^T\mathbb |\eta_s|^2ds\leq r$,  set $X^{\eta,\epsilon}=\Psi^{\epsilon }(\sqrt{\epsilon }B_\cdot
+\int_{0}^{\cdot }\eta_s d\langle B\rangle_s,\langle B\rangle)$. Then
\begin{equation}\label{A1-A2-condition-proof-eq-2}
\begin{aligned}
X^{\eta,\epsilon}(x,t)=&x+\int_{0}^{t}b^{\epsilon}(X^{\eta,\epsilon}(x,s))ds
+\sqrt{\epsilon}\int_{0}^{t}\sigma^{\epsilon}(X^{\eta,\epsilon}(x,s))dB_{s}\\
&+\int_{0}^{t}\sigma^{\epsilon}(X^{\eta,\epsilon}(x,s))\eta_s d\langle B\rangle_s+\int_{0}^{t}h^{\epsilon}(X^{\eta,\epsilon}(x,s))  d\langle B\rangle_s.
\end{aligned}
\end{equation}
and there exists a constant $M=M(\bar{\sigma})$ such that
$$
\left|\int_{0}^{t}\sigma^{\epsilon}(X^{\eta,\epsilon}(x,s))\eta_s d\langle
B\rangle_s\right| \leq
\left(\int_0^t\left|\sigma^{\epsilon}(X^{\eta,\epsilon}(x,s))\right|^2ds\right)^{1/2}
M r^{1/2}.
$$
By the BDG inequality under $G$-expectation  and Gronwall's
equality, we can get  that (cf. \cite{Gao-spa-09}) for $q\geq 2$, for
any $m\geq 1$, there exists a constant    $\beta=\beta(m,q,r,\bar{\sigma})$ such
that
$$
\sup_{\int_0^T\mathbb |\eta_s|^2ds\leq r}\sup_{\epsilon\in
[0,1]}\sup_{|x|\leq m}\mathbb E^G\left(\sup_{t\in [0,T]}
\left|X^{\eta,\epsilon}(x,t)\right|^q\right)\leq \beta
$$
and for any $x,y\in\mathbb R^q$, for any $s,t\in[0,T]$,
\begin{equation}\label{A1-A2-condition-proof-eq-3}
\sup_{\int_0^T\mathbb |\eta_s|^2ds\leq r}\sup_{\epsilon\in
[0,1]}\mathbb
E^G\left(\left|X^{\eta,\epsilon}(x,t)-X^{\eta,\epsilon}(y,s)\right|^q\right)\leq
\beta (|x-y|^q+|s-t|^{q/2}).
\end{equation}
Set
$$
\theta(\epsilon)=\sup_{x\in \mathbb R^p}\max\left\{
|b^{\epsilon}(x)-b(x)|,\|\sigma^{\epsilon}(x)-\sigma(x)\|_{HS},\max_{1\leq k\leq p}\|h^{\epsilon,k}(x)-h^k(x)\|_{HS}
\right\}
$$
and
$Z^{\eta,\epsilon}(x,t)=X^{\eta,\epsilon}(x,t)-X^{\eta,0}(x,t)$. Then
$$
\begin{aligned}
Z^{\eta,\epsilon}
(x,t)=&\sqrt{\epsilon}\int_{0}^{t}\sigma^\epsilon (X^{\eta,\epsilon}(x,s))dB_{s}+\int_{0}^{t}(b^{\epsilon}(X^{\eta,\epsilon}(x,s))-b(X^{\eta,\epsilon}(x,s)))ds\\
&+\int_{0}^{t}(\sigma^{\epsilon}(X^{\eta,\epsilon}(x,s))-\sigma(X^{\eta,\epsilon}(x,s)))\eta_s
d\langle B\rangle_s\\
&+\int_{0}^{t}(h^{\epsilon}(X^{\eta,\epsilon}(x,s))-h(X^{\eta,\epsilon}(x,s)))
d\langle B\rangle_s\\
&\int_{0}^{t}(b (X^{\eta,\epsilon}(x,s))-b(X^{\eta,0}(x,s)))ds\\
&+\int_{0}^{t}(\sigma (X^{\eta,\epsilon}(x,s))-\sigma(X^{\eta,0}(x,s)))\eta_s
d\langle B\rangle_s\\
&+\int_{0}^{t}(h (X^{\eta,\epsilon}(x,s))-h(X^{\eta,0}(x,s)))
d\langle B\rangle_s,
\end{aligned}
$$
and so for any $q\geq 2$, by the BDG inequality under $G$-expectation  and Gronwall's
equality,   there exists a function $ \gamma(\epsilon,\theta(\epsilon),
q,r,\bar{\sigma})$ satisfying $\gamma(\epsilon,\theta(\epsilon), q,r,\bar{\sigma})\to 0$ as $\epsilon\to 0$
such that  (cf. \cite{Gao-spa-09})
$$
\sup_{\int_0^T\mathbb |\eta_s|^2ds\leq r}\sup_{|x|\leq m}\overline{\mathbb
E}^G\left(\sup_{t\in[0,T]}\left|Z^{\eta,\epsilon}(x,t)\right|^q\right)
\leq \gamma(\epsilon,\theta(\epsilon), q,r,\bar{\sigma}),
$$
which yields that
\begin{equation}\label{A1-A2-condition-proof-eq-4}
\lim_{\epsilon\to 0}\sup_{\int_0^T\mathbb |\eta_s|^2ds\leq
r}\sup_{|x|\leq m}\bar{\mathbb E}^G\left(\sup_{t\in[0,T]}
\left|Z^{\eta,\epsilon}(x,t)\right|^q\right)=0.
\end{equation}
Finally, by the below Lemma \ref{kolmogrov-thm-c},
(\ref{A1-A2-condition-proof-eq-1}) is a consequence of
(\ref{A1-A2-condition-proof-eq-3}) and
(\ref{A1-A2-condition-proof-eq-4}).
\end{proof}

For given  $N\geq 1$,  for each $f\in\mathbb H$, $g\in\mathbb A$,   let $ \Psi^{(N)}(f,g)\in C(\mathbb
R^p\times [0,T], \mathbb R^p)$ be defined by
\begin{align*}
&\Psi^{(N)}(f,g)(x,t)\\
=&x+\sum_{k=1}^{N}b\left(\Psi^{(N)}(f,g)\left(x,\frac{(k-1)T}{N}\right)\right)\left(\frac{kT}{N
}\wedge t-\frac{(k-1)T}{N }\wedge t\right)\\
&+\sum_{k=1}^{N}\sigma
\left(\Psi^{(N)}(f,g)\left(x,\frac{(k-1)T}{N}\right)\right)  \left(f\left({\frac{kT}{N}\wedge
t}\right)-f\left({\frac{(k-1)T}{N}\wedge t}\right)\right)\\
&+\sum_{k=1}^{N}h \left(\Psi^{(N)}(f,g)\left(x,\frac{(k-1)T}{N}\right)\right)\left(g\left({\frac{kT}{N}\wedge
t}\right)-g\left({\frac{(k-1)T}{N}\wedge t}\right)\right).
\end{align*}
Then it is obvious that for any  $N\geq 1 $,
$(\mathbb H\times\mathbb A,\rho_{HG})\ni(f,g)\to \Psi^{(N)}(g)$ is continuous and
\begin{align*}
\Psi^{(N)}(f,g)(x,t)
=&x+\int_{0}^{t}b\left(\Psi^{(N)}(f,g)(x,\pi_{N}(s))\right)ds\\
&+\int_{0}^{t}\sigma\left(\Psi^{(N)}(f,g)(x,\pi_{N}(s))\right)df(s)\\
&+\int_{0}^{t}h\left(\Psi^{(N)}(f,g)(x,\pi_{N}(s))\right)dg(s)
\end{align*}
where $\pi_{N}(s)=\frac{(k-1)T}{N}$, for $s\in [(k-1)T/N,kT/N)$, $k=1,\cdots,N$.

\begin{lem}\label{A3-condition-proof}Assume that $(H1) $   and $(H2)$ hold.  Then for any $l\in(0,\infty)$,
$$
\lim_{N\to\infty} \sup_{\|f\|_{H}\leq l,g\in \mathbb A}\rho(\Psi(f,g),\Psi^{(N)}(f,g))=0.
$$

\end{lem}

\begin{proof} Firstly, by the Lipschitz condition, there exists a constant $L_1\in(0,\infty)$ such that  for any $x\in\mathbb R^p$, $t\in [0,T]$,  $f\in\mathbb H$,$g\in\mathbb A$,
$$
\begin{aligned}
|\Psi(f,g)(x,t)|
\leq &|x|+L_1\int_{0}^{t}\left(1+ |\Psi(f,g)(x,s)|\right)\left(1+|f'(s)|+\|g'(s)\|_{HS}\right)ds.
\end{aligned}
$$
Therefore, by Gronwall's inequality, for any $m\geq 1$,
$$
\begin{aligned}
\bar{M}_m:=\sup_{\|f\|_H\leq l, g\in\mathbb A}\sup_{|x|\leq m,t\in [0,T]}|\Psi(f,g)(x,t)|<\infty.
\end{aligned}
$$
Furthermore,  there exist positive constants $L_2,L_3$ such
that for any $|x|\leq m$, $t\in [0,T]$, $\|f\|_H\leq l$, $g\in\mathbb A$,
\begin{align*}
&\left|\Psi(f,g)(x,\pi_{N}(t)) -\Psi(f,g)(x,t)\right|\\
\leq & L_2\max_{1\leq k\leq N}\max_{t\in [(k-1)T/N,kT/N]}\left(\int_{(k-1)T/N}^t |f'(s)|+\|g'(s)\|_{HS}\right)ds\leq  \frac{L_3}{\sqrt{N}}.
\end{align*}
 Therefore,  by
\begin{align*}
&\left|\Psi^{(N)}(f,g)(x,t)-\Psi(f,g)(x,t)\right|\\
\leq&\int_{0}^{t}\left|b\left(\Psi^{(N)}(f,g)(x,\pi_N(s))\right)
-b\left(\Psi(f,g)(x,\pi_N(s))\right) \right|ds\\
&+\int_{0}^{t}\left|b\left(\Psi(f,g)(x,s)\right)
-b\left(\Psi(f,g)(x,\pi_N(s))\right)\right|ds\\
&+\int_{0}^{t}\left|\left(\sigma\left(\Psi^{(N)}(f,g)(x,\pi_N(s))\right)
-\sigma\left(\Psi(f,g)(x,\pi_N(s))\right)\right)f'(s)\right|ds\\
&+\int_{0}^{t}\left|\left(\sigma\left(\Psi(f,g)(x,s)\right)
-\sigma\left(\Psi(f,g)(x,\pi_N(s))\right)\right)f'(s)\right|ds\\
 &+\int_{0}^{t}\left|\left(h\left(\Psi^{(N)}(f,g)(x,\pi_N(s))\right)
-h\left(\Psi(f,g)(x,\pi_N(s))\right)\right)g'(s)\right|ds\\
&+\int_{0}^{t}\left|\left(h\left(\Psi(f,g)(x,s)\right)
-h\left(\Psi(f,g)(x,\pi_N(s))\right)\right)g'(s)\right|ds,
\end{align*}
there exist positive constants $L_4,L_5$ such that  for any $|x|\leq m$, $t\in [0,T]$, $\|f\|_H\leq l$,  $g\in\mathbb A$,
$$
\begin{aligned}
&\max_{s\in[0,t]}\left|\Psi^{(N)}(f,g)(x,s)-\Psi(f,g)(x,s)\right|\\
\leq &\frac{L_4}{\sqrt{N}}+L_5\int_0^t\max_{u\in[0,s]}
\left|\Psi^{(N)}(f,g)(x,u)-\Psi(f,g)(x,u)\right|\left(1+|f'(s)|+\|g'(s)\|_{HS}\right)ds.
\end{aligned}
$$
Therefore,  by Gronwall lemma, we obtain that for any $m\geq 1$,
$$
\lim_{N\to\infty}\sup_{|x|\leq m,t\in[0,T],\|f\|_H\leq l,g\in\mathbb A}\left|\Psi^{(N)}(f,g)(x,t)-\Psi(f,g)(x,t)\right|=0.
$$

\end{proof}

\begin{lem}\label{kolmogrov-thm-c}   Let $T>0$ and let
$\{Y_{\lambda,\epsilon}=\{Y_{\lambda,\epsilon}(t),t\in[0,T]^m\};\epsilon\in[0,1],\lambda\in
\Lambda\}$ be a family  of $\rr^p$-valued continuous processes such
that $Y_{\lambda,\epsilon}(t)$ is quasi-continuous for all
$\lambda,\varepsilon $ and $t$. Assume that there exists constants
$L\in(0,+\infty)$, $q>0$ and $\kappa>0$ such that
\begin{equation} \label{kolmogrov-thm-c-eq-1}
\sup_{\lambda\in
\Lambda,\epsilon\in[0,1]}\mathbb E^G(|Y_{\lambda,\epsilon}(t)-Y_{\lambda,\epsilon}(s)|^{q})\le
C|t-s|^{m+\kappa},\quad s,t\in[0,T]^m.
\end{equation}
Then
\begin{equation} \label{kolmogrov-thm-c-eq-2}
\sup_{\lambda\in
\Lambda,\epsilon\in[0,1]}\overline{\mathbb{E}}^G\left(\left(
\sup_{s\neq t} \displaystyle \frac{|Y_{\lambda,\epsilon}(t)
-Y_{\lambda,\epsilon}(s)|}{|t-s|^{\alpha}}\right) ^{q}\right)
<\infty,
\end{equation}
for every $\alpha \in[0,\kappa/q)$. As a consequence,
$\{\{Y_{\lambda,\epsilon}(t),t\in[0,T]^m\};\epsilon\in[0,1],\lambda\in
\Lambda\}$ is tight under $\mathbb E^G$, i.e., for any $\delta>0$,
there exists a compact $K_\delta\subset C([0,T]^m,\mathbb R^p)$ such
that
\begin{equation} \label{kolmogrov-thm-c-eq-3}
\sup_{\lambda\in
\Lambda,\epsilon\in[0,1]}{c^G}\left(Y_{\lambda,\epsilon}\in
K_\delta^c\right)<\delta.
\end{equation}

Furthermore,  if for $t\in
[0,T]^m$ and any $\delta>0$,
\begin{equation} \label{kolmogrov-thm-c-eq-5}
\lim_{\epsilon\to 0}\sup_{\lambda\in
\Lambda}c^G\left(|Y_{\lambda,\epsilon}(t)-Y_{\lambda}(t)|\geq
\delta\right)=0,
\end{equation}
where $Y_\lambda(t):=Y_{\lambda,0}(t)$, then  $Y_{\lambda,\epsilon}$ converges  uniformly   to $Y_{\lambda}$
in distribution under $\mathbb E^G$, i.e., for any $\Phi\in
C_b(C([0,T]^m,\mathbb R^p))$,
\begin{equation} \label{kolmogrov-thm-c-eq-6}
\lim_{\epsilon\to 0}\sup_{\lambda\in \Lambda}\mathbb
E^G\left(\left|\Phi(Y_{\lambda,\epsilon})-\Phi(Y_{\lambda})\right|\right)=0.
\end{equation}

 \end{lem}

\begin{proof}
First, from the proof of the Kolmogorov criterion under $G$-expectation (cf. Theorem 1.36,
Chapter VI in \cite{peng-book-10})), we can obtain (\ref{kolmogrov-thm-c-eq-2}). Since for each $\alpha\in(0,\kappa/q)$,
$$
\left\{y\in C([0,T]^m,\mathbb R^p);\sup_{s\neq t} \displaystyle \frac{|y(t)
-y(s)|}{|t-s|^{\alpha}}\leq r\right\}
$$
is compact subset for any $r\in (0,\infty)$, by Chebyshev's inequality and (\ref{kolmogrov-thm-c-eq-2}),
for any $\delta>0$,
there exists a compact $K_\delta\subset C([0,T]^m,\mathbb R^p)$ such
that  (\ref{kolmogrov-thm-c-eq-3}) holds.

If
for each $t\in [0,T]^m$ and $\delta>0$,
(\ref{kolmogrov-thm-c-eq-5}) holds. Take $\alpha\in(0,\kappa/q)$. For
any $\delta>0$, choose $r=r(\delta)\in (0,\infty)$ such that
$$
\sup_{\lambda\in
\Lambda,\epsilon\in[0,1]}c^G\left(Y_{\lambda,\epsilon}(t)\in K_r^c\right)<\delta,
$$
where $K_r=\left\{y\in C([0,T]^m,\mathbb R^p);\sup_{s\neq t} \displaystyle \frac{|y(t)
-y(s)|}{|t-s|^{\alpha}}\leq r\right\}$.

By continuity of $\Phi$ and compactness of $K_r$, there exists
$\zeta>0$ such that  for any $\psi,\varphi\in C([0,T]^m,\mathbb R^p)\cap K_r$
with  $\|\psi-\varphi\|\leq\zeta$,
$|\Phi(\psi)-\Phi(\varphi)|<\delta$.

By the definition of $K_r$,
there exist  $l\geq 1$, $\tau\in(0, (\zeta/3)^{1/\alpha}/r)$ and
$t_1,\cdots,t_l$ such that  $[0,T]^m= \cup_{i=1}^l U(t_i,\tau)$,
and
$$
\sup_{\lambda\in \Lambda}c^G\left(\max_{1\leq i\leq l}\sup_{t\in
U(t_i,\tau)}|Y_{\lambda}(t)-Y_{\lambda}(t_i)|\geq
\zeta/3,Y_\lambda\in K_r\right)=0,
$$
where $U(t_i,\tau)=\{t\in[0,T]^m;|t-t_i|<\tau \}$ .

By (\ref{kolmogrov-thm-c-eq-5}), there exists $\epsilon_0$ such that
for all $\epsilon\in(0,\epsilon_0)$,
$$
\max_{1\leq i\leq l} \sup_{\lambda\in
\Lambda}c^G\left(|Y_{\lambda,\epsilon}(t_i)-Y_{\lambda}(t_i)|\geq
\zeta/3\right)<\delta/3,
$$

By triangle inequality
$$
|Y_{\lambda,\epsilon}(t)-Y_{\lambda}(t)|\leq
|Y_{\lambda,\epsilon}(t)-Y_{\lambda,\epsilon}(t_i)|+|Y_{\lambda,\epsilon}(t_i)-Y_{\lambda}(t_i)|+|Y_{\lambda}(t)-Y_{\lambda}(t_i)|,
$$
we have that for all $\epsilon\in(0,\epsilon_0)$,
$$
\sup_{\lambda\in \Lambda}c^G\left(Y_{\lambda,\epsilon}\in
K_r,Y_\lambda\in K_r,\sup_{t\in[0,T]^m}|Y_{\lambda,\epsilon}(t)-Y_{\lambda}(t)|\geq
\zeta\right)<\delta.
$$
Therefore, for all $\epsilon\in(0,\epsilon_0)$,
$
\sup_{\lambda\in
\Lambda}\mathbb E^G\left(\left|\Phi(Y_{\lambda,\epsilon})-\Phi(Y_{\lambda})\right|\right)\leq
(1+3M)\delta,
$
where $M:=\sup_{y\in C([0,T]^m,\mathbb R^p)}|\Phi(y)|$.
This yields (\ref{kolmogrov-thm-c-eq-6}).

\end{proof}


\end{document}